\documentclass[a4paper,11pt]{amsart}
\usepackage{geometry}
\input{latex/header.sty}

\usepackage{soul} 
\setstcolor{red} 
\sethlcolor{Yellow1} 
\newcommand\stout[1]{{\ifmmode\text{\st{$#1$}}\else\st{#1}\fi}} 

\DeclareMathOperator{\red}{red} 

\usepackage[noadjust]{marginnote}

\begin{document}
\title{Commutator width in the first Grigorchuk group}
\author{Laurent Bartholdi}
\author{Thorsten Groth}
\author{Igor Lysenok}
\begin{abstract}
  Let $G$ be the first Grigorchuk group.  We show that the commutator
  width of $G$ is $2$: every element $g\in [G,G]$ is a product of two
  commutators, and also of six conjugates of $a$. Furthermore, we show
  that every finitely generated subgroup $H\leq G$ has finite
  commutator width, which however can be arbitrarily large, and that
  $G$ contains a subgroup of infinite commutator width. The proofs
  were assisted by the computer algebra system GAP.
\end{abstract}
\maketitle

\section{Introduction}
Let $\Gamma$ be a group and let $\Gamma'=[\Gamma,\Gamma]$ denote its
derived subgroup. The \emph{commutator width} of $\Gamma$ is the least
$n\in \NN\cup \infty$ such that every element of $\Gamma'$ is a
product of $n$ commutators.

We compute, in this article, the commutator width of the \emph{first
  Grigorchuk group} $G$, see~\S\ref{ss:bg} for a brief
introduction. This is a prominent example from the class of
\emph{branched groups}, and as such is a good testing ground for
decision and algebraic problems in group theory. We prove:
\begin{thma}\label{thm:CWGrigorchukGroup}
  The first Grigorchuk group and its branching subgroup $K$ have
  commutator width $2$.
\end{thma}
It was already proven in~\cite{Lysenok-Miasnikov-Ushakov:QuadraticEquationsInGrig} that
the commutator width of $G$ is finite, 
without providing an explicit bound.
Our result also
answers a question of Elisabeth
Fink~\cite[Question~3]{Fink:Conjugacy_growth}.
This is closely related to the problem of representing elements of the first Grigorchuk group by products
of conjugates, see~\cite{Fink:Conjugacy_growth}.
\begin{cora}\label{cor:productOf6Conjugates}
  Every element of $G'$ is a product of $6$ conjugates of the
  generator $a$ and there are elements $g\in G'$ which are 
  not products of $4$ conjugates of $a$.
  Every element of G is a product of at most $8$ conjugates of the standard generators $\{a,b,c,d\}$.
\end{cora}

There are examples of groups of finite commutator width with subgroups
of infinite commutator width; and even finitely presented, perfect
examples in which the subgroup has finite index, see
Example~\ref{ex:commwidthsubgroup}. However, we can prove:
\begin{thma}\label{thm:subgroups}
  Every finitely generated subgroup of $G$ has finite commutator
  width; however, their commutator width cannot be bounded, even among
  finite-index subgroups. Furthermore, there is a subgroup of $G$ of infinite commutator width.
\end{thma}

\subsection{Commutator width}
Let $\Gamma$ be a group. It is well-known that usually elements of
$\Gamma'$ are not commutators---for example,
$[X_1,X_2]\cdots[X_{2n-1},X_{2n}]$ is not a commutator in the free
group $F_{2n}$ when $n>1$. In fact, every non-abelian free group has
infinite commutator width, see~\cite{Rhemtulla:CommutatorsF2}.

On the other hand, some classes of groups have finite commutator
width: finitely generated virtually abelian-by-nilpotent
groups~\cite{Segal:Words}, and finitely generated solvable groups of
class $3$, see~\cite{Rhemtulla:Commutators}.

Finite groups are trivial examples of groups of finite commutator
width.  There are finite groups in which some elements of the derived subgroup are not
commutators, the smallest having order $96$,
see~\cite{Guralnick:Group96}. On the other hand, non-abelian finite
simple groups have commutator width $1$, as was conjectured by Ore
in~1951, see~\cite{Ore:Commutators}, and proven in~2010,
see~\cite{Liebeck:OreConjecture}. The commutator width cannot be
bounded among finite groups; for example,
$\Gamma_n=\langle x_1,\dots,x_{2n}\mid
x_1^p,\dots,x_{2n}^p,\gamma_3(\langle x_1,\dots,x_{2n})\rangle$ is a
finite class-$2$ nilpotent group in which $\Gamma_n'$ has order
$p^{\binom{2n}2}$ but at most $\binom{p^{2n}}2$ elements are
commutators, so $\Gamma_n$'s commutator width is at least $n/2$.

Commutator width of groups, and of elements, has proven to be an
important group property, in particular via its connections with
``stable commutator length'' and bounded
cohomology~\cite{Calegari:SCL}. It is also related to solvability of
quadratic equations in groups: a group $\Gamma$ has commutator width
$\le n$ if and only if the equation
$[X_1,X_2]\cdots[X_{2n-1},X_{2n}]g=\id$ is solvable for all
$g\in\Gamma'$. Needless to say, there are groups in which solvability
of equations is algorithmically undecidable. It was proven
in~\cite{Lysenok-Miasnikov-Ushakov:QuadraticEquationsInGrig} that 
there exists an algorithm to check solvability of quadratic equations  in the first Grigorchuk group.

We note that if the character table of a group $\Gamma$ is computable,
then it may be used to compute the commutator width: Burnside shows
(or, rather, hints) in~\cite[\S238, Ex. 7]{Burnside:Groups} that an
element $g\in\Gamma$ may be expressed as a product of $r$ commutators
if and only if
\[\sum_{\chi\in\operatorname{Irr}(\Gamma)}\frac{\chi(g)}{\chi(1)^{2r-1}}>0.
\label{eq:BurnsideFormula} \]
This may yield another proof of Theorem~\ref{thm:CWGrigorchukGroup},
using the quite explicit description of $\operatorname{Irr}(G)$ given
in~\cite{Bartholdi:RepresentationZetaFunctions}.

Consider a group $\Gamma$ and a subgroup $\Delta$. There is in general
little connection between the commutator width of $\Gamma$ and that of
$\Delta$. If $\Delta$ has finite commutator width and
$[\Gamma:\Delta]$ is finite, then obviously $\Gamma$ also has finite
commutator width---for example, because
$\Gamma/\operatorname{core}(\Delta)'$ is virtually abelian, and every
commutator in $\Gamma$ can be written as a product of a commutator in
$\Delta$ with the lift of one in
$\Gamma/\operatorname{core}(\Delta)'$, but that seems to be all that
can be said. Danny Calegari pointed to us the following example:
\begin{ex}\label{ex:commwidthsubgroup}
  Consider the group $\Delta$ of orientation-preserving
  self-homeomorphisms of $\RR$ that commute with integer translations,
  and let $\Gamma$ be the extension of $\Delta$ by the involution
  $x\mapsto-x$. Then, by~\cite[Theorems~2.3
  and~2.4]{Eisenbud-Hirsch-Neumann:SeifertBundles}, every element of
  $\Gamma'=\Delta$ is a commutator in $\Gamma$, while the commutator
  width of $\Delta$ is infinite.

  Both $\Gamma$ and $\Delta$ can be made perfect by replacing them
  respectively with $(\Gamma\wr A_5)'$ and $\Delta\wr A_5$; and can be
  made finitely presented by restricting to those self-homeo\-mor\-phisms
  that are piecewise-affine with dyadic slopes and breakpoints.
\end{ex}

\subsection{Branched groups}\label{ss:bg}
We briefly introduce the first Grigorchuk
group~\cite{Grigorchuk:Burnside} and some of its properties. For a more
detailed introduction into the topic of self-similar groups we refer
to~\cite{Bartholdi-Grigorchuk-Sunik:BranchGroups,Nekrashevych:SelfSimilarGroups} and
to Section~\ref{sec:SelfSimilarGroups}.

A \emph{self-similar group} is a group $\Gamma$ endowed with an
injective homomorphism $\Psi\colon\Gamma\to\Gamma\wr S_n$ for some
symmetric group $S_n$. It is \emph{regular branched} if there exists a
finite-index subgroup $K\le\Gamma$ such that $\Psi(K)\ge K^n$. It is
convenient to write $\pair{g_1,\dots,g_n}\pi$ for an element
$g\in\Gamma\wr S_n$. We call $g_i$ the \emph{states} of $g$ and $\pi$
its \emph{activity}. It is also convenient to identify, in a
self-similar group, elements with their image under $\Psi$.

Note that, by definition, every group may be viewed as self-similar;
the self-similarity is an attribute of a group, not a
property. However, being regular branched imposes strong conditions on
the group.

A self-similar group may be specified by giving a set $S$ of
generators, some relations that they satisfy, and defining $\Psi$ on
$S$. There is then a maximal quotient $\Gamma$ of the free group $F_S$
on which $\Psi$ induces an injective homomorphism to $\Gamma\wr S_n$.

The first Grigorchuk group $G$ may be defined in this manner. It is
the group generated by $S=\{a,b,c,d\}$, with $a^2=b^2=c^2=d^2=bcd=\id$,
and with
\[a=\pair{\id,\id}(1,2),\quad b=\pair{a,c},\quad c=\pair{a,d},\quad d=\pair{\id,b}. \]

Here are some remarkable properties of $G$: it is an infinite torsion
group, and more precisely for every $g\in G$ we have $g^{2^n}=\id$ for
some $n\in\NN$. On the other hand, it is not an Engel group, namely it
is not true that for every $g,h\in G$ we have $[g,h,\dots,h]=\id$ for
a long-enough iterated commutator~\cite{Bartholdi:Engel}. It is a
group of intermediate word growth~\cite{Grigorchuk:Milnor}, and
answered in this manner a celebrated question of Milnor.
For more information about the Grigorchuk group, see the extensive survey~\cite{Grigorchuk:Survey}.

We have decided to concentrate on the first Grigorchuk group in the
computational aspects of this text; though our code would function
just as well for other examples of self-similar branched groups, such
as the Gupta-Sidki groups~\cite{Gupta-Sidki:PGroups}.


\subsection{Sketch of proofs}
The general idea for the proof of Theorem~\ref{thm:CWGrigorchukGroup}
is the decomposition of group elements into states via $\Psi$. We show
that each element $g\in G'$ is a product of two commutators by solving
the equation $\eq{E}=[X_1,X_2]\cdots[X_{2n-1},X_{2n}]g$ for all
$n\geq 2$.

If there is a solution then the values of the variables $X_i$ have
some activities $\sigma_i$. If we fix a possible activity of the
variables of $\eq E$ then by passing to the states of the $X_i$ we are
led to two new equations which (under mild assumptions and after some
normalization process) yields a single equation of the same form but
of higher genus.

Not all solutions for the new equations lead back to solutions of the
original equation. Thus instead of pure equations we consider
\emph{constrained} equations: we require the variables to lie in
specified cosets of the finite-index subgroup $K$. The pair composed
of a constraint and an element $g\in G$ will be a \emph{good pair} if
there is some $n$ such that the constrained equation
$[X_1,X_2]\cdots[X_{2n-1},X_{2n}]g$ is solvable.  It turns out that
this only depends on the image of $g$ in the finite quotient
$\RNK{G}{K'}$.

Then by direct computation we show that every good pair leads to
another good pair in which the genus of the equation increases.  We
build a graph of good pairs which turns out to be finite since the
constants of the new equation are states of the old equation and we
can use the strong contracting property of $G$.

The computations could in principle be done by hand, but one of our
motivations was precisely to see to which point they could be
automated. We implemented them in the computer algebra system
GAP~\cite{GAP4}. The source code for these computations is distributed
with this document as ancillary material. It can be validated using
precomputed data on a GAP standard installation by running the command
\gapinline{gap verify.g} in its main directory.

To perform more advanced experimentation with the code and to recreate
the precomputed data, the required version of GAP must be at least
$4.7.6$ and the packages FR~\cite{FR2.3.6} and LPRES~\cite{LPRES0.3.0}
must be installed.

\section{Equations}
We fix a set $\mathcal{X}$ and call its elements \emph{variables}.  We
assume that $\mathcal{X}$ is infinite countable, is well ordered, and
that its family of finite subsets is also well ordered, by size and
then lexicographic order. We denote by $F_{\mathcal{X}}$ the free
group on the generating set $\mathcal{X}$. We use $\id$ for the
identity element of groups, and for the identity maps, to distinguish
it from the numerical $1$.

\begin{defi}[$G$-group, $G$-homomorphism]
  Let $G$ be a group. A \emph{$G$-group} is a group with a
  distinguished copy of $G$ inside it; a typical example is 
  $H*G$ for some group $H$. A \emph{$G$-homomorphism} 
  between $G$-groups is a homomorphism
  that is the identity between the marked copies of $G$.

  A \emph{$G$-equation} is an element $\eq{E}$ of the $G$-group
  $F_{\mathcal{X}} * G$, regarded as a reduced word in
  $\mathcal X\cup\mathcal X^{-1}\cup G$. For $\eq{E}$ a $G$-equation,
  its set of \emph{variables} $\Var(\eq{E})\subset \mathcal{X}$ is the
  set of symbols in $\mathcal{X}$ that occur in it; namely,
  $\Var(\eq{E})$ is the minimal subset of $\mathcal{X}$ such that
  $\eq{E}$ belongs to $F_{\Var(\eq{E})}*G$.

  An \emph{evaluation} is a $G$-homomorphism $e\colon F_{\mathcal{X}} * G \to G$.
  A \emph{solution} of an equation $\eq{E}$ is an evaluation $s$
  satisfying $s(\eq{E})=\id$. If a solution exists for $\eq{E}$ then the
  equation $\eq{E}$ is called \emph{solvable}. The set of elements
  $X\in \mathcal{X}$ with $s(X)\neq\id$ is called the \emph{support} of the solution.
\end{defi}

The support of a solution for an equation $\eq{E}$ may be assumed to be
a subset of $F_{\Var(\eq{E})}$ and hence the data of a solution
is equivalent to a map $\Var(\eq{E}) \to G$.  The question of whether an
equation $\eq{E}$ is solvable will be referred to as the \emph{Diophantine
problem} of $\eq{E}$.

Every homomorphism $\varphi \colon G \to H$ extends uniquely to an
$F_{\mathcal{X}}$-homomorphism
$\varphi_* \colon F_{\mathcal{X}}*G \to F_{\mathcal{X}}*H$.  In this
manner, every $G$-equation $\eq{E}$ gives rise to an $H$-equation
$\varphi_*(\eq{E})$, which is solvable whenever $\eq{E}$ is solvable.

\begin{defi}[Equivalence of equations]
  Let $\eq{E},\eq{F}\in F_{\mathcal{X}}* G$ be two $G$-equations. We
  say that $\eq{E}$ and $\eq{F}$ are \emph{equivalent} if there is a
  $G$-automorphism $\varphi$ of $F_{\mathcal{X}}*G$ that maps $\eq{E}$
  to $\eq{F}$. We denote by $\Stab(\eq E)$ the group of
  all $G$-automorphisms that fix $\eq E$.
\end{defi}
\begin{lem}
  Let $\eq{E}$ be an equation and let $\varphi$ be a $G$-endomorphism of
  $F_{\mathcal{X}}*G$. If $\varphi(\eq{E})$ is solvable then so is $\eq{E}$. In particular,
  the Diophantine problem is the same for equivalent equations.
\end{lem}
\begin{proof}
  If $s$ is a solution for $\varphi(\eq{E})$, then $s\circ\varphi$ is a
  solution for $\eq{E}$.
\end{proof}

\subsection{Quadratic equations}
A $G$-equation $\eq{E}$ is called \emph{quadratic} if for each variable
$X\in \Var(\eq{E})$ exactly two letters of $\eq{E}$ are $X$ or $X^{-1}$, when
$\eq{E}$ is regarded as a reduced word.

A $G$-equation $\eq{E}$ is is called \emph{oriented} if for each variable
$X\in \Var(\eq{E})$ the number of occurrences with positive and with
negative sign coincide, namely if $\eq{E}$ maps to the identity under the
natural map $F_{\mathcal{X}}*G\to F_{\mathcal{X}}/[F_{\mathcal{X}},F_{
\mathcal{X}}]*\id$.
Otherwise $\eq{E}$ is called \emph{unoriented}.
\begin{lem}
 Being oriented or not is preserved under equivalence of equations.
\end{lem}
\begin{proof}
  $\eq{E}$ is oriented if and only if it belongs to the normal closure of
  $[F_{\mathcal{X}},F_{\mathcal{X}}]*G$; this subgroup is preserved by all $G$-endomorphisms
  of $F_{\mathcal{X}}*G$.
\end{proof}

\subsection{Normal form of quadratic equations} \label{sec:normal_form}
\begin{defi}[$\eq{O}_{n,m}, \eq{U}_{n,m}$]
  For $m,n\ge0$, $X_i,Y_i,Z_i \in\mathcal{X}$ and $c_i \in G$ the following two kinds of
  equations are called in \emph{normal form}:
 \begin{align}
  \eq{O}_{n,m}:\qquad & [X_1,Y_1][X_2,Y_2]\cdots[X_n,Y_n]c_1^{Z_1}\cdots c_{m-1}^{Z_{m-1}}c_m  \\
   \eq{U}_{n,m}:\qquad & X_1^2X_2^2\cdots X_n^2 c_1^{Z_1}\cdots c_{m-1}^{Z_{m-1}}c_m\ .
 \end{align} 
 The form $\eq{O}_{n,m}$ is called the oriented case and $\eq{U}_{n,m}$ for
 $n>0$ the unoriented case.  The parameter $n$ is referred to as
 the \emph{genus} of the normal form of an equation.
\end{defi}

We recall the following result, and give the details of the proof in
an algorithmic manner, because we will need them in practice:
\begin{thm}[\cite{Comerford-Edmunds:EquationsFreeGroups}] \label{Thm:equationNormalForm}
  Every quadratic equation $\eq{E} \in F_{\mathcal{X}}*G$ is equivalent to an equation
  in normal form, and the $G$-isomorphism can be effectively computed.
\end{thm}

\begin{proof}
  The proof proceeds by induction on the number of variables.
  Starting with the oriented case: if the reduced equation $\eq{E}$ has no
  variables then it is already in normal form $\eq{O}_{0,1}$. If there is a
  variable $X\in\mathcal{X}$ occurring in $\eq{E}$ then $X^{-1}$ also appears.
  Therefore the equation has the form
  $\eq{E} = uX^{-1}vXw$ or can be brought to this form by applying the
  automorphism $X \mapsto X^{-1}$. Choose $X\in\mathcal{X}$ in such a way that
  $\Var(v)$ is minimal.
 
  We distinguish between multiple cases:
  \begin{itemize}
  \item[Case $1.0$:] $v\in G$. The word $uw$ has fewer variables than
    $\eq{E}$ and can thus be brought into normal form $r\in \eq{O}_{n,m}$ by a
    $G$-isomorphism $\varphi$. If $r$ ends with a variable, we use
    the $G$-isomorphism $\varphi \circ (X\mapsto Xw^{-1}) $ to map $\eq{E}$
    to the equation $rv^X \in \eq{O}_{n,m+1}$.   
    If $r$ ends with a group constant $b$, say $r=sb$, we use the
    isomorphism $\varphi \circ(X \mapsto Xbw^{-1}) $ to map $\eq{E}$ to the
    equation $sv^Xb\in \eq{O}_{n,m+1}$.

  \item[Case $1.1$:] $v\in\mathcal{X}\cup\mathcal X^{-1}$. For simplicity let us assume
    $v\in\mathcal{X}$; in the other case we can apply the $G$-homomorphism
    $v \mapsto v^{-1}$.
    Now there are two possibilities: either $v^{-1}$ occurs in $u$ or
    $v^{-1}$ occurs in $w$. In the first case $\eq{E}= u_1v^{-1}u_2X^{-1}vXw$, and
    then the $G$-isomorphism $X \mapsto X^{u_1}u_2$, $v \mapsto v^{u_1}$
    yields the equation $[v,X]u_1u_2w$. In the second case
    $\eq{E}= uX^{-1}vXw_1v^{-1}w_2$ is transformed to $[X,v]uw_1w_2$ by the
    $G$-isomorphism $X \mapsto X^{uw_1}w_1^{-1}$, $v\mapsto v^{-uw_1}$. In
    both cases $u_1u_2w$, respectively $uw_1w_2$ have fewer variables
    and so composition with the corresponding $G$-isomorphism results in a
    normal form.
  \item[Case $2$:] Length$(v)>1$. In this case $v$ is a word consisting of
    elements from $\mathcal X\cup \mathcal X^{-1}$ with each symbol occurring at most once as
    $v$ was chosen with minimal variable set, and some elements of
    $G$.  If $v$ starts with a constant $b\in G$ we use the
    $G$-homomorphism $X\mapsto bX$ to achieve that $v$ starts with a
    variable $Y\in\mathcal{X}$, possibly by using the $G$-homomorphism 
    $Y \mapsto Y^{-1}$. As in Case
    $1.1$ there are two possibilities: $Y^{-1}$ is either part of $u$
    or part of $w$. In the first case $\eq{E}= u_1 Y^{-1} u_2 X^{-1}Yv_1Xw$
    we can use the $G$-isomorphism $X\mapsto X^{u_1v_1}u_2$,
    $Y\mapsto Y^{u_1v_1}v_1^{-1}$ to obtain $[Y,X]u_1v_1u_2w$. In the
    second we use the $G$-isomorphism
    $X\mapsto X^{uw_1v_1}v_1^{-1}w_1^{-1}$,
    $Y\mapsto Y^{-uw_1v_1}v_1^{-1}$ to obtain $[X,Y]uw_1v_1w_2$. In
    both cases the second subword has again fewer variables and can be
    brought into normal form by induction.
  \end{itemize}
  Therefore each oriented equation can be brought to normal form by 
  $G$-iso\-mor\-phisms.

  In the unoriented case there is a variable $X\in\mathcal{X}$ such that
  $\eq{E} = uXvXw$. Choose $v$ to have a minimal number of variables.
  By induction, the shorter word
  $uv^{-1}w$ is equivalent by $\varphi$ to a normal form $r$.
 
  The $G$-isomorphism $\varphi \circ (X\mapsto X^uv^{-1})$ maps $\eq{E}$ to
  $X^2r$. If $r\in \eq{U}_{n,m}$ for some $n,m$,
  there remains nothing to do.  Otherwise $r=[Y,Z]s$, and then the
  $G$-homomorphism
  \begin{align*}
    X&\mapsto XYZ, & Y&\mapsto Z^{-1}Y^{-1}X^{-1}YZXYZ, & Z&\mapsto Z^{-1}Y^{-1}X^{-1}Z \\
  \intertext{maps $X^2r$ to $X^2Y^2Z^2s$. This homomorphism is indeed an
  isomorphism, with inverse}
    X&\mapsto X^2Y^{-1}X^{-1}, & Y&\mapsto XYX^{-1}Z^{-1}X^{-1}, & Z&\mapsto XZ.
  \end{align*}
  Note that $s \in \eq{O}_{n,m}$. If $n\geq 1$ then this procedure can be repeated with
  $Z,$ in place of $X,r$.
\end{proof}
For a quadratic equation $\eq{E}$ we denote by $\nf(\eq{E}) := \nf_{\eq E}(\eq{E})$
the image of $\eq{E}$ under the $G$-isomorphism $\nf_{\eq E}$ constructed in
the proof.

From now on we will consider oriented equations $\eq{O}_{n,1}$. For this
we will use the abbreviation
\[R_n(X_1,\dotsc,X_{2n})=\prod_{i=1}^n [X_{2i-1},X_{2i}]\]
and often write $R_n=R_n(X_1,\dotsc,X_{2n})$ if the $X_i$ are the
first generators of $F_{\mathcal{X}}$.

\subsection{Constrained equations}
\begin{defi}[Constrained equations~\cite{Lysenok-Miasnikov-Ushakov:QuadraticEquationsInGrig}]
  Given an equation $\eq{E} \in F_{\mathcal{X}}*G$, a group $H$ with a fixed
  homomorphism $\pi\colon G \to H$ and a homomorphism
  $\gamma\colon F_{\mathcal{X}} \to H$, the pair $(\eq{E},\gamma)$ is
  called a \emph{constrained} equation and $\gamma$ is called a
  \emph{constraint} for the equation $\eq{E}$ on $H$.
 
  A \emph{solution} for $(\eq{E},\gamma)$ is a solution $s$ for
  $\eq{E}$ with the additional property that $\pi\circ s=\gamma$.
\end{defi}
We note that the constraint $\gamma$ needs only to be specified on $\Var(\eq E)$.

\section{Self-similar groups}\label{sec:SelfSimilarGroups}
Let $T_n$ be the regular rooted $n$-ary tree and let $S_n$ be the symmetric group on $n$ symbols.
The group $\Aut(T_n)$ consists of all root-preserving graph automorphisms of the tree $T_n$. 

Let $T_{1,n},\dotsc,T_{n,n}$ be the subtrees hanging from neighbors of the root. 
Every $g\in\Aut(T_n)$ permutes the $T_{i,n}$ by a permutation $\sigma$ and simultaneously
acts on each of them by isomorphisms $g_i\colon T_{i,n}\to T_{i^\sigma,n}$.

Note that for all $i$ the tree $T_n$ is isomorphic to $T_{i,n}$;
identifying each $T_{i,n}$ with $T_n$, we identify each $g_i$ with an
element of $\Aut(T_n)$, and obtain in this manner an isomorphism
\[\Psi\colon\left\{\begin{array}{r@{\;}l}
              \Aut(T_n) &\xrightarrow{\sim} \Aut(T_n)\wr S_n\\
              g &\mapsto \pair{g_1,\dotsc,g_n}\sigma.
            \end{array}\right.
\]

A \emph{self-similar group} is a subgroup $G$ of $\Aut(T_n)$
satisfying $\Psi(G)\le G\wr S_n$.  For the sake of notation we will
identify elements with their image under this embedding and will write
$g=\pair{g_1,\dotsc,g_n}\sigma$ for elements $g\in G$.  Furthermore we
will call $g_i\in G$ the \emph{states} of the element $g$, will write
$g\at{i}:=g_i$ to address the states, will call $\sigma \in S_n$ the
\emph{activity} of the element $g$, and will write $\act(g):= \sigma$.

\subsection{Commutator width of \texorpdfstring{$\mathbf{\textup{Aut}(T_2)}$}{Aut(T2)} }
To give an idea of how the commutator width of Grigorchuk's group is
computed, we consider as an easier example the group $\Aut(T_2)$. In
this group we have the following useful property: for every two
elements $g,h\in \Aut(T_n)$ the element $\pair{g,h}$ is also a member
of the group.  This is only true up to finite index in the Grigorchuk
group and will produce extra complications there.

\begin{pro}\label{pro:comwidthAutT2}
 The commutator width of $\Aut(T_2)$ is $1$.
\end{pro}
For the proof we need a small observation:
\begin{lem}\label{lem:H'}
  Let $H$ be a self-similar group acting on a binary tree.
  If $g\in H'$ then $g\at{2}\cdot g\at{1}\in H'$. 
\end{lem}
\begin{proof}
  It suffices to consider a commutator $g=[g_1,g_2]$ in $H'$. Then
  $g\at 2\cdot g\at 1$ is the product, in some order, of all eight
  terms $(g_i\at j)^\epsilon$ for all $i,j\in\{1,2\}$ and
  $\epsilon\in\{\pm1\}$.
\end{proof}

%

\begin{proof}[Proof of Proposition~\ref{pro:comwidthAutT2}]
  Given any element $g\in \Aut(T_2)'$ we consider the equation
  $[X,Y]g$.  If in it we replace the variable $X$ by $\pair{X_1,X_2}$
  and $Y$ by $\pair{Y_1,Y_2}(1,2)$ we obtain
  $\pair{X_1^{-1}Y_2^{-1}X_2Y_2(g\at{1}),X_2^{-1}Y_1^{-1}X_1Y_1(g\at{2})}$.
  Therefore, $[X,Y]g$ is solvable if the system of equations
  $\{X_1^{-1}Y_2^{-1}X_2Y_2(g\at{1}),X_2^{-1}Y_1^{-1}X_1Y_1(g\at{2})\}$
  is solvable.  We apply the $\Aut(T_2)$-homomorphism
  $X_1\mapsto X_1,X_2\mapsto Y_1^{-1}X_1Y_1 (g\at{2}),Y_i\mapsto Y_i$
  to eliminate one equation and one variable.
 
  Thus the solvability of the constrained equation
  $\left([X,Y]g, (X\mapsto\id,Y\mapsto (1,2)) \right)$ follows from
  the solvability of
  $X_1^{-1}Y_2^{-1}Y_1^{-1}X_1Y_1(g\at{2}) Y_2 (g\at{1})$ which is
  under the normal form $\Aut(T_2)$-isomorphism
  $Y_1\mapsto Y_1Y_2^{-1}$ equivalent to the solvability of
  $[X_1,Y_1](g\at{2})^{Y_2} g\at{1}$. After choosing $Y_2=\id$ we are
  again in the original situation since $g\at{2} g\at{1}\in H'$.
 
  This allows us to recursively define a solution $s$ for the equation
  $[X,Y]g$ as follows:
  \begin{align*}
    s(X) &= \pair{a_1,b_1^{-1}a_1b_1 g\at{2} },
    & s(Y) &= \pair{b_1,\id}(1,2), &c_1&=g\at{2} \cdot g\at{1}, \\
\intertext{and for all $i\geq 1$}
    a_{i} &= \pair{a_{i+1},b_{i+1}^{-1}a_{i+1}b_{i+1} c_i\at{2} }, & b_{i} &=\pair{b_{i+1},\id}(1,2),&c_{i+1}&=c_{i}\at{2} \cdot c_{i}\at{1}. 
  \end{align*}
  Note that the elements $a_i,b_i\in\Aut(T_2)$ are well-defined,
  although they are constructed recursively out of the $a_j,b_j$ for
  \emph{larger} $j$. Indeed, if one considers the recursions above for
  $i\in\{1,\dots,n\}$ and sets $a_{n+1}=b_{n+1}=\id$, one defines in
  this manner elements $a_1^{(n)},b_1^{(n)}\in\Aut(T_2)$ which form
  Cauchy sequences, and therefore have well-defined limits
  $a_1=\lim a_1^{(n)}$ and $b_1=\lim b_1^{(n)}$.
\end{proof}

\section{The first Grigorchuk Group}\label{sec:GrigorchukGroup}
The first Grigorchuk group~\cite{Grigorchuk:Burnside} is a finitely
generated self-similar group acting faithfully on the binary rooted
tree, with generators
\[a=\pair{\id,\id}(1,2),\quad b=\pair{a,c},\quad c=\pair{a,d},\quad d=\pair{\id,b}. \]

Some useful identities are
\begin{gather*}
  a^2=b^2=c^2=d^2=bcd=\id,\\
  b^a= \pair{c,a}, c^a=\pair{d,a}, d^a=\pair{b,\id},\\
  (ad)^4=(ac)^8=(ab)^{16}= \id.
\end{gather*}
\begin{defi}[Regular branched group]
 A self-similar group $\Gamma$ is called \emph{regular branched} if it
 has a finite-index subgroup $K\leq \Gamma$ such that $K^{\times n} \leq \Psi(K)$.
\end{defi}
\begin{lem}[\cite{Rozhkov:Centralizers}]\label{lem:subgroupK}
  The Grigorchuk group is regular branched with branching subgroup 
  \[K:= \left<(ab)^2\right>^G=\left< (ab)^2,(bada)^2,(abad)^2
    \right>.
  \]
  The quotient $Q := \RNK{G}{K}$ has order $16$.\qed
\end{lem}

For an equation $\eq E\in F_{\mathcal X}*G$, recall that
$\Stab(\eq E)$ denotes the group of $G$-auto\-mor\-phisms of $\eq E$.

Denote by $U_n$ the subgroup of $\Stab(R_n)$ generated by the
following automorphisms of $F_{2n}$:
\[\begin{array}{rr@{\;}ll}
    \varphi_i\colon& X_i&\mapsto X_{i-1}X_i, \textup{ others fixed} & \textup{for }i=2,4,\dotsc,2n,\\[2ex]
    \varphi_i\colon& X_i&\mapsto X_{i+1}X_i, \textup{ others fixed} & \textup{for }i=1,3,\dotsc,2n-1,\\[2ex]
    \multirow{5}{*}{$\psi_i\colon$\hbox to 0pt{$\left\{\rule{0mm}{13mm}\right.$}}& X_i&\mapsto X_{i+1}X_{i+2}^{-1}X_i, & \multirow{5}{*}{for $i=1,3,\dotsc,2n-3$.}\\[1ex]
                   & X_{i+1}&\mapsto X_{i+1}X_{i+2}^{-1}X_{i+1}X_{i+2}X_{i+1}^{-1},\\[1ex]
                   & X_{i+2}&\mapsto X_{i+1}X_{i+2}^{-1}X_{i+2}X_{i+2}X_{i+1}^{-1},\\[1ex]
                   & X_{i+3}&\mapsto X_{i+1}X_{i+2}^{-1}X_{i+3}, \textup{ others fixed}
\end{array}\]

\begin{re}
  In fact, we have $U_n = \Stab(R_n)$ though formally we do not need
  the equality.  Due to classical results of Dehn--Nielsen,
  $\Stab(R_n)$ is isomorphic to the mapping class groups $M(n,0)$ of
  the closed orientable surface of genus $n$. It can be checked that
  the automorphisms $\varphi_i$ and $\psi_i$ represent the Humphries
  generators of $M(n,0)$.  For details on mapping class groups, see
  for example~\cite{Farb-Margalit:MCG}.
\end{re}

\begin{lem}[\cite{Lysenok-Miasnikov-Ushakov:QuadraticEquationsInGrig}]\label{lem:finitelyManyConstraints}
  Given $n\in \NN$ and a homomorphism
  $\gamma\colon F_{\mathcal{X}} \to Q$ with
  $\supp(\gamma) \subset \left<X_1,\ldots ,X_{2n}\right>$ there is an
  element $\varphi \in U_n<\Aut(F_{\mathcal{X}})$ such that
  $\supp(\gamma\circ\varphi) \in \left<X_1,\dotsc,X_5\right>$.\qed
\end{lem}
\begin{lem}\label{lem:90Constraints}
 Identify the set $\{\gamma \colon F_{\mathcal{X}} \to Q \mid \supp(\gamma)
 \subset \left<X_1,\ldots, X_n\right>\}$ with $Q^{n}$. 
 Then 
 \[\left|\RNK{Q^{2n}}{U_n}\right|\leq 90 \text{ for all } n \ge 3. \]
\end{lem}
\begin{proof}
  Note that according to our identification we have $Q^m \subset Q^n$
  for $m < n$.  By Lemma~\ref{lem:finitelyManyConstraints} every orbit
  $Q^{2n} / U_n$ has a representative in $Q^5$.  
Then $|Q^{2n}/U_n| = |Q^5/U_n|$ and since $U_n \subset U_{n+1}$ we have
$|G^{2n+2}/U_{n+1}| \le |G^{2n}/U_n|$. Direct computation gives $|G^6/U_3| = 90$, see Section~\ref{sec:gap_details}.

\end{proof}
\begin{re}
  It can be proved by an extra computation that indeed
  $\left|\RNK{Q^{2n}}{U_n}\right|= 90$ for all
  $n\geq3$.
\end{re}

\begin{nota}[$\Red$, reduced constraint]
  Lemmas~\ref{lem:finitelyManyConstraints} and~\ref{lem:90Constraints}
  imply that there is a set of $90$ homomorphisms
  $\gamma\colon F_{\mathcal{X}} \to Q$ with
  $\supp(\gamma) \subset \left<X_1,\ldots, X_5\right>$ that is a
  representative system of the orbits $Q^{2n} / U_n$ for each
  $n \ge 3$.  
  Note that representatives are formally not assumed unique if $n \ge 4$
  (though in fact they are unique according to the remark above).
  Fix such a set $\Red$ and for
  $\gamma \colon F_{\mathcal{X}} \to Q$ with finite support (say
  $X_1,\dots,X_{2n}$) denote by $\varphi_\gamma$ the $G$-automorphism
  in $U_n$ such that $\gamma\circ \varphi_\gamma \in \Red$.
 
  The element $\gamma \circ \varphi_\gamma$ will be called a
  \emph{reduced constraint}, denoted $\red(\gamma)$.
We extend the function $\red(*)$ also to the case when $\gamma: F_{\mathcal{S}} \to Q$
is a homomorphism defined on any finite subset $\mathcal{S}$ of variables from $\mathcal{X}$:
we simply extend $\gamma$ onto $F_{\mathcal{X}}$ by defining $\gamma(X) = \id$ for $X \notin \mathcal{S}$
and then take $\red(\gamma)$ as already defined.

\begin{re}
Considering finitely supported constraints defined on an infinite set of variables $\mathcal X$ is a convenient trick
that allows us to compare constraints intended for equations with different number of variables. 
In particular, we will assert in Section~\ref{sec:succ_pairs} that certain sets of constraints are
independent on the number of variables of an equation.
\end{re}

\end{nota}

\begin{lem} \label{lem:solvabilityWithReducedConstraint} The
  solvability of a constrained equation $(R_n g,\gamma)$ is equivalent
  to the solvability of $(R_n g,\gamma\circ \varphi_\gamma)$.
\end{lem}
\begin{proof}
  If $s$ is a solution for $(R_n g,\gamma)$ then
  $s\circ \varphi_\gamma$ is a solution for
  $(R_n g,\gamma\circ \varphi_\gamma)$ and vice versa.
\end{proof}

\begin{defi}[Branch structure~\cite{Bartholdi:RepresentationZetaFunctions}] 
  A \emph{branch structure} for a group $G\hookrightarrow G \wr S_n$
  consists of
  \begin{enumerate}
  \item a branching subgroup $K\normal G$ of finite index;
  \item the corresponding quotient $Q=\RNK{G}{K}$ and the factor homomorphism $\pi\colon G \to Q$;
  \item a group $Q_1 \subset Q \wr S_n$ such that
    $\pair{q_1,\ldots,q_n}\sigma\in Q_1$ if and only if
    $\pair{g_1,\ldots,g_n}\sigma \in G$ for all
    $g_i \in \pi^{-1}(q_i)$;
  \item a map $\omega\colon Q_1 \to Q$ with the following property: if
    $g=\pair{g_1,\ldots,g_n}\sigma \in G$ then
    $\omega(\pair{\pi(g_1),\ldots,\pi(g_n)}\sigma) = \pi(g)$.
\end{enumerate}
\end{defi}
All regular branched groups have a branch structure (see~\cite[Remark
  after Definition~5.1]{Bartholdi:RepresentationZetaFunctions}).  We
will from now on fix such a structure for $G$ and take the group $K$
defined in Lemma~\ref{lem:subgroupK} as branching subgroup and denote
by $Q$ the factor group with natural homomorphism
$\pi\colon G \to \RNK{G}{K}=Q$.

\begin{re}
  The branch structure of $G$ is included in the FR package and can be
  computed by the method \gapinline{BranchStructure(GrigorchukGroup)}.
\end{re}

\subsection{Good Pairs}\label{sec:good_pairs}
It is not true that for every $g\in G'$ and every constraint $\gamma$
there is an $n\in\NN$ such that the constrained equation
$(R_ng,\gamma)$ is solvable. For example
\[\left(R_n(ab)^2,(\gamma\colon X_i\mapsto\id\;\forall i)\right)\] 
is not solvable for any $n$ because $(ab)^2\notin K'$.  This motivates
the following definition.
\begin{defi}[Good pair]
Given $g\in G'$ and $\gamma\in\Red$, the tuple $(g,\gamma)$ is called a \emph{good pair} if 
$(R_ng,\gamma)$ is solvable for some $n\in\NN$.  
\end{defi}

\begin{lem}
  For $g\in G$, let $\overline g$ denote the image of $G$ in
  $\RNK{G}{K'}$. Then the pair $(g,\gamma)$ is a good pair if and only
  if $(R_3 \overline g,\gamma)$ has a solution in $\RNK{G}{K'}$.
\end{lem}
\begin{proof}
  If $(g,\gamma)$ is a good pair and $s$ a solution for
  $(R_ng,\gamma)$ then $s(X_i)\in K$ for all $i\geq6$, so
  $s(R_n g) = s(R_3) \cdot k g$ for some $k\in K'$.  Therefore
  $\overline{s(R_3)}\overline{k}\overline{g}=\overline{s(R_3)}\overline{g}=\id$,
  and $(R_3\overline g,\gamma)$ has a solution
  $\overline s\colon X_i\mapsto\overline{s(X_i)}$. Clearly
  $\overline s$ satisfies the constraint $\gamma$, since
  $G\twoheadrightarrow Q$ factorizes through $\RNK{G}{K'}$.
 
  Now suppose that $(R_3\overline g,\gamma)$ has a solution
  $\overline s\colon F_{\mathcal{X}} \to \RNK{G}{K'}$, so we have
  $\overline s(R_3\overline g)=\overline s(R_3)\overline g=\id$ in
  $\RNK{G}{K'}$. There are then $k\in K'$ and $g_1,\dots,g_6\in G$
  with $R_3(g_1,\dots,g_6)k g=\id$, so $(g,\gamma)$ is a good pair.
\end{proof}

The previous lemma shows that the question whether $(g,\gamma)$ is a
good pair depends only on the image of $g$ in $\RNK{G}{K'}$. For
$q\in Q$, we call $(q,\gamma)$ a \emph{good pair} if $(g,\gamma)$ is a
good pair for one (and hence all) preimages of $q$ under $G\twoheadrightarrow\RNK{G}{K'}$.
\begin{cor}\label{cor:finiteCommutatorWidthKimpliesBoundedConstraintedCommutators}
  The following are equivalent:
  \begin{enumerate}[(a)]
  \item $K$ has finite commutator width; \label{Cor:EqStatement1}
  \item there is an $n\in\NN$ such that $(R_n g,\gamma)$ is solvable 
    for all good pairs $(g,\gamma)$ with $g \in G'$ and $\gamma\in\Red$.
    \label{Cor:EqStatement2}
  \end{enumerate} 
\end{cor}
\begin{proof}
  (\ref{Cor:EqStatement2})$\Rightarrow$(\ref{Cor:EqStatement1}): if
  $k\in K'$ then $(k,\id)$ is a good pair, so $(R_nk,\id)$ is solvable
  in $G$ for some $n$; and the constraints ensures that it is solvable in $K$.
  Therefore the commutator width of $K$ is at most $n$.
 
  (\ref{Cor:EqStatement1})$\Rightarrow$(\ref{Cor:EqStatement2}): if
  $(g,\gamma)$ is a good pair there is an $m'\in \NN$ and a solution
  $s$ for $(R_{m'} g,\gamma)$. As $\pi(s(X_i)) =\id$ for all $i\geq 6$
  there is $k\in K'$ such that $s$ is a solution for
  $(R_3kg,\gamma)$. By (\ref{Cor:EqStatement1}) there is an $m$ such
  that all $k$ can be written as product of $m$ commutators of
  elements of $K$ and therefore there is a solution for
  $(R_{m+3}g,\gamma)$. We may take $n=m+3$.
\end{proof}

We study now more carefully the quotients $\RNK{G}{K}$, $\RNK{G}{K'}$
and $\RNK{G}{(K\times K)}$.
\begin{lem} \label{lem:subgroupsOfG}
Let us write $k_1 := (ab)^2, k_2:=\pair{\id,k_1} = (abad)^2$ and $k_3:=\pair{k_1,\id} = (bada)^2$. Then 
 \begin{align*}
  G' &= \left< k_1,k_2,k_3,(ad)^2\right>, \\
  K &= \left< k_1,k_2,k_3\right>, \\
  K\times K &= \{\pair{k,k'} \mid k,k'\in K\} \\&= \left< k_2,k_3,[k_1,k_2],[k_1,k_3],[k_1^{-1},k_2],[k_1^{-1},k_3] \right>,\\
  K' &= \left< [k_1,k_2] \right>^G \\ 
  &= \left<[k_2,k_1],[k_1,k_2^{-1}],[k_2,k_1]^{k_2},[k_1^{-1},k_2], [k_2,k_1]^{k_1},[k_2^{-1},k_1^{-1}]\right>^{\{\id,a\}} 
 \end{align*}
Furthermore these groups form a tower with indices
\begin{align*}
  [G:G']&=8, & [G':K]&=2, &[K:K\times K]&= 4, &[K\times K:K']&=16. 
\end{align*}
\end{lem}
\begin{proof}
 The chain of indices is shown for example in~\cite{Bartholdi-Grigorchuk-Sunik:BranchGroups} 
 and the generating sets can be verified using the GAP standard methods
 \gapinline{NormalClosure} and \gapinline{Index}. 
\end{proof}

\subsection{Succeeding pairs} \label{sec:succ_pairs}
The main step in our proof is a procedure that accepts as input a good
pair $(g,\gamma)$ and produces a ``succeeding pair'' $(g',\gamma')$ in
such a manner that solvability of $(R_n g,\gamma)$ is equivalent to that
of $(R_{n'} g',\gamma')$ for some $n' > n$. The procedure, while completely explicit (and
actually implemented) is quite complicated, and involves the
construction (via sets $\Gamma_1^{\cdots}$, $\Gamma_2^{\cdots}$,
$\Gamma_3^{\cdots}$ and $\Gamma_4^{\cdots}$) of a non-empty set
$\Gamma^q(\gamma)$ of admissible succeeding pairs, with $q\in G'/K'$
representing $g$, from which $\gamma'$ will be appropriately
chosen. The reader is forewarned that this section is the most
technical.

\begin{defi}[$\Red_{\act}$, active constraints]
  We define the activity $\act(q)$ of an element $q\in Q$ as the
  activity of an arbitrary element of $\pi^{-1}(q)$.  This is well
  defined since all elements of $K$ have trivial activity.
 
  Consider a constraint $\gamma\colon F_{\mathcal{X}} \to Q$.  Define
  $\act(\gamma)\colon F_{\mathcal{X}}\to C_2$ by
  $X \mapsto \act(\gamma(X))$. 
 
  Denote by $\Red_{\act}$ the reduced constraints in $\Red$ that have a
  nontrivial activity.
\end{defi}

\begin{lem} \label{lem:existsGoodGamma} For each $q\in\RNK{G'}{K'}$
  there is $\gamma\in \Red_{\act}$ such that $(q,\gamma)$ is a good
  pair.
\end{lem}
\begin{proof}
  This is a finite problem which can be checked in GAP with the
  function \gapinline{verifyLemmaExistGoodConstraints}.  For more
  details see Section~\ref{sec:gap_verify}.
\end{proof}

We will now give a procedure that starts with a constrained equation
of class $\mathcal{O}_{n,1}$ and produces a finite family of constrained equations of class
$\mathcal{O}_{2n-1,1}$. Because we need to specialize equations inside the 
procedure, the reduction is one-side: solvability of any equation of the family implies
solvability of the initial one. We prove the reverse way reduction in a weaker form 
utilizing the notion of a good pair (Proposition~\ref{pro:existsNextPair} below)
which will be enough for our purposes (and actually implies equivalence of the initial
equation and the produced family).

The idea of the procedure is to replace each variable $X_\ell$ of the starting equation $(R_ng,\gamma)$ 
by two variables $Y_{\ell,1}$ 
and $Y_{\ell,2}$ representing the states of $X_\ell$, so 
$X_\ell=\pair{Y_{\ell,1},Y_{\ell,2}}\act(X_\ell)$, and then transform the resulting system 
of two equations to a quadratic equation in the standard form.
We denote $\mathcal{Y} = \{ Y_{\ell,i} \mid \ell \ge 1,\ i=1,2 \}$ the target set of variables
and $F_{\mathcal{Y}}$ the free group with basis $\mathcal{Y}$.
After the transformation we obtain a set of equations of the form $(R_{2n-1}g',\gamma')$ 
where $(g',\gamma')$ runs over a certain finite set
and $R_{2n-1}$ is written in variables $\{Y_{\ell,i} \mid 1\leq \ell \leq 2n,\ \ell \ne 6, i=1,2 \}$. 
%

In what follows, we will define
for all $q\in \RNK{G'}{K'}$ a map $\Gamma^q$ which maps constraints to finite
sets of constraints,
\[\Gamma^q\colon(\gamma\colon F_{\mathcal X} \to Q)\mapsto\big\{(\gamma'\colon F_{\mathcal Y}\to Q)\big\},\]
with the following property: 
\begin{itemize}
\item[(*)] For every constraint $\gamma'\in\Gamma^q(\gamma)$, there is
  $x\in\{\id,a,b,c,d,ab,ad,ba\}\subset G$ 
such that if $g G' = q$ then for any $n\ge 3$ the equation $(R_ng,\gamma)$ is solvable as soon as the
  constrained equation
  $(R_{2n-1} (g\at{2})^x \cdot g\at{1},\gamma')$
  is solvable.
\end{itemize}

We will define this map $\Gamma^q$ in several steps and afterwards
show that for all good pairs $(q,\gamma)$ and all $g$ with $g K'=q$
there is some constraint $\gamma' \in \Gamma^q(\gamma)$ such that
$((g\at{2})^x \cdot g\at{1},\gamma'|_{F_{\mathcal{Y}'}})$ is a good
pair. The first step is to construct a set $\Gamma_1(\gamma)$ of
constraints on the ``doubled'' alphabet $\mathcal Y$ that are lifts of
$\gamma$. The second step extracts from $\Gamma_1(\gamma)$ a subset
$\Gamma_2^{q_1,q_2}(\gamma)$ of constraints compatible with a target
$\pair{q_1,q_2}\in Q\times Q$. The third step rewrites elements of
$\Gamma_2^{q_1,q_2}(\gamma)$ in normal form using a letter $Y_0$,
defining a set $\Gamma_3^{q_1,q_2,Y_0}(\gamma)$ of reduced
constraints. The fourth step combines these constraints over all
possible $Y_0$ into a set $\Gamma_4^{q_1,q_2}(\gamma)$, and the last
step extracts from $\Gamma_4^{q_1,q_2}(\gamma)$ a set
$\Gamma^q(\gamma)$, with $q=\pair{q_1,q_2}$ in $G'/K'$, by requiring
some activity to be non-trivial and lie in a specific subset of $Q$.

We assume that some $n\ge 3$ is fixed. It will be straightforward to see from the 
construction at each step that the corresponding set $\Gamma_i^{\cdots}$ does not depend on $n$.

For the first step we take the branching structure
$(K,Q,\pi,Q_1,\omega)$ of the Grigorchuk group. Set
\[\Gamma_1(\gamma) = \left \{\gamma' 
    \colon F_{\mathcal Y} \to Q \bigmid
    \begin{array}{l}
      \omega(\pair{\gamma'(Y_{\ell,1}),\gamma'(Y_{\ell,2})}\act(X_\ell))=\gamma(X_\ell)\text{ if }1\le\ell\le 6,\\
      \gamma'(Y_{\ell,1}) = \gamma'(Y_{\ell,2}) = \id\text{ if }\ell>6
    \end{array} \right\}.
\]
For some formal equalities for equations in $G$ we will need two auxiliary free
groups $F_{\mathcal{G}}=\left<\mathfrak{g}\right>$,
$F_{\mathcal{H}}=\left<\mathfrak{g}_1,\mathfrak{g}_2\right>$, and define  homomorphisms  
\[\Phi_\gamma\colon\left\{\begin{array}{rl}
 F_{\mathcal{X}}*F_{\mathcal G} &\to (F_{\mathcal{Y}}*F_{\mathcal{H}})\wr C_2,\\
  \mathfrak g   &\mapsto \pair{\mathfrak g_1, \mathfrak g_2},\\ 
  X_i &\mapsto \pair{Y_{i,1},Y_{i,2}}\act(X_i),
 \end{array}\right.\quad
\tilde\Phi_\gamma\colon\left\{\begin{array}{rl}
 F_{\mathcal{X}}*G &\to (F_{\mathcal{Y}}*G)\wr C_2,\\
   g   &\mapsto \Psi(g),\\ 
  X_i &\mapsto \pair{Y_{i,1},Y_{i,2}}\act(X_i).
 \end{array}\right.
 \]
\begin{lem}\label{lem:commonVar}
  If $\gamma$ is a constraint with nontrivial activity, and
  $\Phi_\gamma(R_n\mathfrak g)=\pair{w_1,w_2}$ then
  $\Var(w_1)\cap\Var(w_2)\neq \emptyset$. 
\end{lem}
\begin{proof}
  Let $\ell\in 1\ldots 2n$ be such that $\gamma(X_\ell)$ has
  nontrivial activity. Then $R_n$ contains either a factor $[X_\ell,X_k]$ or
  $[X_k,X_\ell]$ for another generator $X_k\neq X_\ell$.  Assume without loss of
  generality the first case. Let $\sigma$ be the activity of $\gamma(X_k)$. Then $\Phi_\gamma(R_n \mathfrak g)$
  contains a factor
  \[ [\pair{Y_{\ell,1},Y_{\ell,2}}(1,2),\pair{Y_{k,1},Y_{k,2}}\sigma] = 
  \begin{cases}
      \pair{Y_{\ell,2}^{-1}Y_{k,2}^{-1}Y_{\ell,2}Y_{k,1},Y_{\ell,1}^{-1}Y_{k,1}^
      {-1}Y_{\ell,1}Y_{k,2}}  \textup{\;\,if } \sigma=\id \\
      \pair{Y_{\ell,2}^{-1}Y_{k,1}^{-1}Y_{\ell,1}Y_{k,2},Y_{\ell,1}^{-1}Y_{k,2}^
      {-1}Y_{\ell,2}Y_{k,1}}  \textup{\;\,if } \sigma=(1,2). \\
    \end{cases} \!\!\!\!\!\! \]
  In both cases $Y_{k,1},Y_{k,2}\in \Var(w_1)\cap \Var(w_2)$.
\end{proof}

\noindent For $q_1,q_2\in Q$ define $\theta\colon F_{\mathcal H}\to Q$ by $\mathfrak g_i\mapsto q_i$ for $i=1,2$; 
if $\gamma'\colon F_{\mathcal Y}\to Q$ is a constraint, 
we denote by $\gamma'*\theta$ the natural map $F_{\mathcal Y}*F_{\mathcal H}\to Q$ 
agreeing with $\theta$ and $\gamma'$ on the respective factors, 
and by $(\gamma'*\theta)^2$ the induced map $(F_{\mathcal Y}*F_{\mathcal H})\wr C_2\to Q\wr C_2$. 
Then define the following subset of $\Gamma_1(\gamma)$: 
\begin{equation}\label{eq:gamma2}
  \Gamma_2^{q_1,q_2}(\gamma) = \left\{\gamma'\in \Gamma_1(\gamma) \bigmid
    (\gamma'*\theta)^2(\Phi_\gamma(R_n \mathfrak g))=\pair{\id,\id}\right\}.
\end{equation}

For $\gamma\in\Red_{\act}$ denote by $v$ and $w$ the elements of
$F_{\{Y_ {1,1},\ldots,Y_{6,2}\}}$ such that
$\Phi_\gamma(R_3 \mathfrak g)=\pair{v,w}\pair{\mathfrak g_1, \mathfrak
  g_2}$.  Then, since the variables $X_7,\dots,X_{2n}$ have trivial activity,
\[\Phi_\gamma(R_n(X_*) \mathfrak g)=\pair{v,w}
\pair{R_{n-3}(Y_{7,1},\ldots,Y_{2n,1})\mathfrak g_1,
  R_{n-3}(Y_{7,2},\ldots,Y_{2n,2})\mathfrak g_2}.
\]
By Lemma~\ref{lem:commonVar} there is
$Y_0 \in \mathcal{Y}\cup \mathcal{Y}^{-1}$ such that $v=v_1Y_0v_2$ and
$w=w_1Y_0^{-1}w_2$. Next, we improve the form of
$\Phi_\gamma(R_n(X_*)\mathfrak g)$, without affecting the constraint
$\gamma'$, by means of the $F_{\mathcal{H}}$-homomorphism
\[\ell_{Y_0}\colon\left\{\begin{array}{r@{\;}l}
                    F_{\mathcal{Y}}*F_{\mathcal H}&\to F_{\mathcal{Y}}*F_
                    {\mathcal H},\\
                    Y &\mapsto \begin{cases}
    Y &\text{ if } Y\neq {Y_0} \\
    w_2R_{n-3}(Y_{7,2},\ldots,Y_{2n,2})\mathfrak g_2 w_1&\text{ if }Y= {Y_0} 
  \end{cases}\end{array}\right.
\]
that eliminates the variable $Y_0$. It maps the second coordinate of
$\Phi_\gamma(R_n(X_*)\mathfrak{g})$ to $\id$ and the first coordinate
to a quadratic equation over $G * \mathcal{H}$
\[\eq E=v_1w_2R_{n-3}(Y_{7,2},\ldots,Y_{2n,2})\mathfrak g_2w_1v_2R_
  {n-3}(Y_{7,1},\ldots,Y_{2n,1}) \mathfrak g_1.
\]
Moreover, from $\gamma'\in \Gamma_2^{q_1,q_2}(\gamma)$ we get
\[
\id=(\gamma'*\theta)(w R_{n-3}(Y_{7,2}, \dots
Y_{2n,2})\mathfrak{g_2})=(\gamma'*\theta)(w_1Y_0^{-1}w_2R_{n-3}(Y_{7,2},
\dots Y_{2n,2})\mathfrak{g_2}).
\]

Thus we obtain
\[
(\gamma'*\theta)(w_2R_{n-3}(Y_{7,2}, \dots
Y_{2n,2})\mathfrak{g_2}w_1)=(\gamma'*\theta)(Y_0)
\] 
and $\ell_{Y_0}$
does not affect the constraint $\gamma'$. 

From this we conclude 
$(\gamma'*\theta)(Y_0)=(\gamma'*\theta)(\ell_{Y_0}(Y_0))$. Since
$\ell_{Y_0}$ fixes all $Y\neq Y_0$ we see that in fact

\begin{equation}
  \gamma'*\theta=(\gamma'*\theta)\circ \ell_{Y_0}\textup{ for all }\gamma'\in
  \Gamma_2^{q_1,q_2}(\gamma) \text{ with }\theta\colon \mathfrak{g}_i\mapsto q_i.
\label{eq:gamma'invariantUnderEll}
\end{equation} 
Consider the automorphisms
\begin{align*}
  \psi_1\colon&\left\{\begin{array}{r@{\;}ll}
                 F_{\mathcal{Y}}*F_{\mathcal{H}} &\to F_{\mathcal{Y}}*F_
                 {\mathcal H}\\
                 Y_{k,1} &\mapsto Y_{k,1}^{\mathfrak g_1^{-1}}&\text{ for } k> 6,\\
                 Y_{k,2} &\mapsto Y_{k,2}^{(\mathfrak g_2w_1v_2\mathfrak g_1)^{-1}}&\text{ for } k> 6,\\
                 Y_{k,\ell} &\mapsto Y_{k,\ell}&\text{ for } k\leq 6,\ \ell=1,2,
              \end{array}\right.\\
  \psi_2\colon&\left\{\begin{array}{r@{\;}ll}
                 F_{\mathcal{Y}}*F_{\mathcal{H}} &\to F_{\mathcal{Y}}*F_
                 {\mathcal{H}}\\
                 Y_{k,1} &\mapsto Y_{k,1}^{\mathfrak g_2^{Y_{6,1}}\mathfrak g_1}&\text{ for } k> 6,\\
                 Y_{k,2} &\mapsto Y_{k,2}^{\mathfrak g_2^{Y_{6,1}}\mathfrak g_1}&\text{ for } k> 6,\\
                 Y_{k,\ell} &\mapsto Y_{k,\ell} &\text{ for } k\leq 6,\ \ell=1,2,
               \end{array}\right.\\
\intertext{ and }
  \psi_3\colon&\left\{\begin{array}{r@{\;}l@{\quad}l}
                 F_{\mathcal{Y}}*F_{\mathcal{H}} &\to F_{\mathcal{Y}}*F_
                 {\mathcal H}\\
                 Y_{2k,1} &\mapsto Y_{n+k,2}&\text{for } k>3,\\
                 Y_{2k-1,1} &\mapsto Y_{n+k,1}&\text{for } k>3,\\
                 Y_{2k,2} &\mapsto Y_{3+k,2}&\text{for } k>3,\\
                 Y_{2k-1,2} &\mapsto Y_{3+k,1}&\text{for } k>3,\\
                 Y_{k,\ell} &\mapsto Y_{k,\ell}&\text{for } k\leq6,\ \ell=1,2;\\
              \end{array}\right.
\end{align*}
For an equation $E$, let $\nf_E$ denote a transformation in $\Aut(F_{\mathcal{X}})$ that puts $E$ in the normal form.
Then $\nf_{\eq E} := \psi_3\circ\psi_2\circ\nf_{v_1w_2\mathfrak g_2w_1v_2\mathfrak g_1}\circ\psi_1$ 
does this for our equation $\eq E$. Indeed,
first $\psi_1$ groups the terms $v_1w_2\mathfrak g_2 w_1 v_2\mathfrak g_1$ at the beginning; 
then $\nf_{v_1w_2\mathfrak g_2w_1v_2\mathfrak g_1}$ puts these terms into the form 
$[Y_{1,1},Y_{1,2}]\cdots\mathfrak g_2^{Y_{6,1}}\mathfrak g_1$, 
and finally $\psi_2$ and $\psi_3$ reorder and renumber the variables. Therefore
\[\nf_{\eq E}(\eq E)
=R_{2n-1}(Y_{1,1},Y_{1,2},\ldots,\widehat{Y_{6,1}},\widehat{Y_{6,2}},\ldots,Y_
{2n,2})\mathfrak{g}_2^{Y_{6,1}}\mathfrak{g}_1.\]
This leads to the following definition.
\[\Gamma_3^{q_1,q_2,Y_0}(\gamma) = 
  \left\{\gamma' \circ \nf_{\eq E}^{-1}\colon F_{\mathcal{Y}}\to
    Q \bigmid \gamma' \in \Gamma_2^{q_1,q_2} \right\}.
\]
Note that $\nf_{\eq E}$ fixes the sets
$\{Y_{k,\ell} \mid k>6,\ell=1,2\}$ and
$\{Y_{k,\ell} \mid k\leq 6,\ell=1,2\}$ and hence for $k>6$ we have
$\gamma''(Y_{k,\ell})=\id$ for all
$\gamma''\in\Gamma_3^{q_1,q_2,Y_0}(\gamma)$ independently of $q_i,Y_0$
and $\gamma$. The set $\Gamma_3^{q_1,q_2,Y_0}$ is therefore
independent of $n$ as soon as $n\geq 3$.  

Denote by $S$ the set $\{\id,a,b,c,d,ab,ad,ba\}\subset G$, and
complete it to a transversal $S'$ of $K$ in $G$. For $q\in Q$, denote
by $\rep(q)\in S'$ the coset representative of $q$, Given $g\in G'$,
$g_i = g\at{i}$ for $i=1,2$, an active constraint
$\gamma\in\Red_ {\act}$ and
$\gamma''\in \Gamma_3^{\pi(g_1),\pi(g_1),Y_0}(\gamma)$ then a solution
for the constrained equation
\[\eq E'=(R_{2n-1}(Y_{*,*})g_2^{\rep(\gamma''({Y_{6,1}}))}
  g_1,\gamma'')
\]
can be extended by the map
${Y_{6,1}}\mapsto \rep(\gamma''({Y_{6,1}}))$ to a solution $s'$ of the
equation $(R_{2n-1}(Y_{*,*}) g_2^{{Y_{6,1}}} g_1,\gamma'')$.  Denote
the homomorphism
$i_{\mathcal{H}}\colon F_{\mathcal H} \to G, \mathfrak g_i \mapsto
g_i$ and note that since $\nf_{\eq E}$ is an
$F_{\mathcal{H}}$-homomorphism, the function
$(\id * i_{\mathcal H}) \circ \nf_{\eq E}$ maps $\eq E$ to
$R_{2n-1}(Y_{*,*}) g_2^ {{Y_{6,1}}}g_1$. Moreover by
\eqref{eq:gamma'invariantUnderEll} we have that
$\gamma' := \gamma''\circ(\id * i_{\mathcal H}) \circ
\nf_{\eq E}\in\Gamma_2^{q_1,q_2}(\gamma)$, so the map
\[s \colon Y_{i,j}\mapsto \begin{cases}
w_2 g_2w_1 &\textup{ if } 
{i,j}= 6,2 \\
s'\circ (\id * i_{\mathcal H}) \circ \nf_{\eq E}(Y_{i,j}) &\textup{ otherwise}
\end{cases}
\]
is a solution for $\big((\id*i_{\mathcal{H}})\circ\Phi_\gamma(R_n\mathfrak g),
\gamma')$ and thus also for $\big(\tilde\Phi_\gamma(R_n g),\gamma'\big)$.
By the definition of $\omega$ the element
$t_i:=\pair{s(Y_{i,1}),s(Y_{i,2})}\act(X_i)$ belongs to $G$
for all $i$. Moreover since $\gamma'\in\Gamma_1(\gamma)$ we have $\pi
(t_i)=\gamma(X_i)$. Thus the mapping
$X_i\mapsto t_i$
is a solution for $(R_ng,\gamma)$.

The map $\Gamma_3^{q_1,q_2,Y_0}$ does depend on the choice of the
variable $Y_0$.  To remove this dependency we observe that the set of
all variables $Y_0 \in \Var(v)\cap \Var(w)$ does not depend on $n$ and
define
\[\Gamma_4^{q_1,q_2}(\gamma)=  \bigcup_{Y_0\in \Var(v)\cap\Var(w)}\Gamma_3^{q_1,q_2,Y_0}(\gamma).\]

Note that $q_{1},q_2\in Q$ are determined by $q\in\RNK{G'}{K'}$ in the
sense that there is a map $\att{i}\colon \RNK{G'}{K'}\to Q$ such that
if $g\in G$ and $g K'=q$ and $g_i=g\at{i}$ then $q_i = q\att{i}$. This
map $\att{i}$ is well defined since $k'\at{i}\in K$ for all
$k'\in K'$. Thus we can write $\Gamma_4^{q_1,q_2}(\gamma)$ as
$\Gamma_4^q(\gamma)$.
Denote 
$
  \mathcal{V} = \{Y_{\ell,i} \mid 1\leq \ell \leq 2n,\ \ell \ne 6,\ i=1,2 \}
$
the set of variables that occur in $R_{2n-1}(Y_{*,*})$.
Filtering out those constraints that do not
fulfill the requested properties we finally define
\begin{align}
  \Gamma^q(\gamma) := \left\{\gamma' \in \Gamma_4^q(\gamma) \bigmid \,
  \act(\gamma')|_\mathcal{V}\neq\id, \gamma'({Y_{6,1}}) \in \pi(S) \right\}
  \label{def:Gammaq}
\end{align}
Note that~(*) holds automatically by construction.
It is straightforward to check that the set $\Gamma^q(\gamma)$ does not depend on $n \ge 3$.

Now we track solutions of equations in the reverse way.
\begin{pro}\label{pro:existsNextPair}
  For each good pair $(q,\gamma)$ with $q\in\RNK{G'}{K'}$ and
  $\gamma\in\Red_{\act}$ the set $\Gamma^q(\gamma)$ contains some
  constraint $\gamma'$ such that for all $g\in G'$ with $g K'=q$ the
  pair
  $\left((g\at{2})^{\rep(\gamma'(Y_{6,1}))}\cdot g\at{1},
  \red(\gamma'|_{F_{\mathcal{V}}})\right)$ is a good pair.
\end{pro}

\noindent For the proof of this proposition we need an auxiliary
lemma:
\begin{lem} \label{lem:pIsDefinedModK'}
  The map 
  \[\overline p_h\colon\left\{\begin{array}{rl}\RNK{G'}{K'} &\to \RNK{G'}{(K\times K)}\\
     gK'&\mapsto \big((g\at{2})^h\cdot g\at{1}\big)(K\times K)
                  \end{array}\right.
  \]
  is well defined.
\end{lem}
\begin{proof}
  We need to show that $k\at{i} \in K\times K$ for $i=1,2$ and
  $k\in K'$. Remember the generators $k_1=(ab)^2$,
  $k_2=(abad)^2$. Then
  \begin{align*}
    [k_1,k_2]=bb^a(db^a)^2b^ab(b^ad)^2 = \pair{\id,cabab}=\pair{\id,\pair{\id,dabac}}=\pair{\id,\pair{\id,k_2^{-1}k_1}}.
  \end{align*}
  Therefore, both states of $[k_1,k_2]$ are in $K\times K$. Now take
  an arbitrary element $k\in K'$.  There are $n\in\NN$,
  $\varepsilon \in \{1,-1\}$ and $g_i\in G$ such that
  $k=\prod_{j=1}^n [k_1,k_2]^{\varepsilon g_j}$ and therefore
  \[k\at{i}=\prod_{j=1}^n\left( \left([k_1,k_2]^{\varepsilon g_j}\right)\at{i}\right)
    = \prod_{j=1}^n \left(([k_1,k_2])\at{i^{g_j^{-1}}}\right)^{\varepsilon g_j\at{i^{g_j^{-1}}}}\in K\times K.\]
  Define for $h\in G$ maps $p_h\colon G\to G$ by
  $g\mapsto (g\at{2})^{h}\cdot g\at{1}$. These maps are in general not
  homomorphisms, but by Lemma~{\ref{lem:H'}} for $g\in G'$ we have
  $p_h(g)\in G'$ for all $h\in G$.
  
  For $k\in K'$ we have
  \[p_h(gk) = ((gk)\at{2})^h\cdot (gk)\at{1} = (g\at{2})^h\cdot (k\at{2})^h\cdot g\at{1}\cdot k\at{1}\in \big((g\at{2})^h\cdot g\at{1}\big)(K\times K).\]
\end{proof}
\begin{proof}[Proof of Proposition~\ref{pro:existsNextPair}]
  In the construction above it is clear that the sets
  $\Gamma_3^{q,Y_0}$ and hence $\Gamma_4^{q}$ are nonempty. For
  the finitely many $\gamma\in \Red_{\act}$ checking whether some of
  the finitely many $\gamma'\in\Gamma_4^{q}(\gamma)$ fulfill
  $\gamma'(Y_{6,1}) \in \pi(S)$ and $\act(\gamma')|_{\mathcal{V}}\neq\id$ (i.e.\
  $\gamma' \in\Gamma^{q}(\gamma)$) is implemented in the procedure
  below.
 

  By Lemma~\ref{lem:pIsDefinedModK'}, we have a map
  $\overline p_h\colon \RNK{G'}{K'} \to \RNK{G'}{(K\times K)}$. For
  $g\in\RNK{G'}{K'}$ let us denote by $\overline g$ the natural image
  of $g$ in
  $\RNK{\left(\RNK{G'}{K'}\right)}{\left(\RNK{(K\times K)}{K'}\right)}
  \simeq \RNK{G'}{(K\times K)}$. We only need to show that there is a
  $\gamma'\in\Gamma^q(\gamma)$ such that all preimages of
  $\overline p_{\rep(\gamma'(Y_{6,1}))}(q)$ under $g\mapsto\overline g$
  form good pairs with $\red(\gamma'|_{F_{\mathcal{V}}})$. In formulas with
  $\mathcal{P}$ the predicate of being a good pair what needs to be
  checked is:
\begin{multline*}
 \forall q\in\RNK{G'}{K'}\;
      \forall \gamma\in\Red_{\act} \;
	 \exists \gamma'\in \Gamma^q(\gamma)\; \\
	    \forall r\in\RNK{G'}{K'}\text{ with }\overline r=\overline p_{\rep(\gamma'(Y_{6,1}))}(q)\colon
	      \mathcal{P}(q,\gamma) \Rightarrow \mathcal{P}(r,\red(\gamma'|_{F_{\mathcal{V}}})).
\end{multline*}
 
 This last formula quantifies only over finite sets, and could be implemented.
 It can be checked in GAP with the function 
 \gapinline{verifyPropExistsSuccessor}. 
\end{proof}

\begin{defi}[Succeding pair]\label{def:succedingPair}
  For each $q\in\RNK{G'}{K'}$ and $\gamma\in \Red_{\act}$ such that
  $(q,\gamma)$ is a good pair fix a constraint
  $\gamma'\in\Gamma^q(\gamma)$ and an element
  $x=\rep(\gamma'(Y_{6,1}))\in S$ with the property of Proposition
  \ref{pro:existsNextPair}.

  By Lemma~\ref{lem:solvabilityWithReducedConstraint} we can replace
  $\gamma'|_{F_{\mathcal{V}}}$ by a reduced constraint $\gamma'_r$.
  Since $\act(\gamma')|_{\mathcal{V}}\neq\id$ we have $\gamma'_r \in \Red_{\act}$.
  For a good pair $(g,\gamma)\in G'\times\Red_{\act}$ the
  \emph{succeeding pair} is defined as
  $\left((g\at{2})^xg\at{1},\gamma'_r\right)$.  Moreover by applying
  this iteratively we get the \emph{succeeding sequence}
  $(g_k,\gamma_k)$ of $(g,\gamma)$: $(g_0,\gamma_0)=(g,\gamma)$ and
  $(g_{k+1},\gamma_{k+1})$ is the succeding pair of $(g_k,\gamma_k)$.
\end{defi}

The following lemma illustrates the use of the construction.

\begin{lem}
  Let $(g_k,\gamma_k)$ be the succeeding sequence of a good pair
  $(g,\gamma)$.  If $(g_i,\gamma_i) = (g_j,\gamma_j)$ for some
  distinct $i,j$ then the equation $(R_ng,\gamma)$ is solvable for all
  $n \ge 3$.
\end{lem}
  \begin{proof}
  By (*) for any $i,j$ with $i < j$ and any $n \ge 3$ there exists $n' > n$ such that solvability of 
  $(R_{n'}g_j,\gamma_j)$ implies solvability of $(R_{n}g_i,\gamma_i)$. If $(g_i,\gamma_i) = (g_j,\gamma_j)$
  then starting from index $i$ the succeeding sequence becomes periodic and hence
  $n'$ can be taken arbitrarily large. If $(g,\gamma)$ is a good pair then $(g_i,\gamma_i)$ is also a good pair
  by construction. We deduce the solvability of $(R_n g_i,\gamma_i)$ and hence the solvability of $(R_n g,\gamma)$.
  \end{proof}

 
%
%
\subsection{Product of 3 commutators}
We will prove that every element $g\in G'$ is a product of three commutators by proving that all
succeeding sequences $(g_k,\gamma_k)$ as defined in Definition~\ref{def:succedingPair} 
become periodic after finitely many steps.
For this purpose remember the map $p_x\colon g \mapsto (g\at{2})^x g\at{1}$ from the proof of Proposition~\ref{pro:existsNextPair}.
We will show that for each $g\in G'$ the sequence of sets 
\[\textup{Suc}_1^g=\{g\},\ \textup{Suc}_n^g = \{p_x(h) \mid h\in \textup{Suc}_{n-1}^g, x\in S\} \]
stabilizes in a finite set. 

In~\cite{Bartholdi:Growth} there is a choice of weights on generators which result in a length on $G$ with good properties.
\begin{lem}[\cite{Bartholdi:Growth}] \label{lem:laurentsweights}\pushQED{\qed}
 Let $\eta\approx 0.811$ be the real root of $x^3+x^2+x-2$ and set the weights 
 \begin{align*}
  \omega(a) &= 1-\eta^3 & \omega(c)&=1-\eta^2 \\ \omega(b)&= \eta^3 & \omega(d)&=1-\eta
 \end{align*}
 then 
 \begin{align*}
  \eta(\omega(b)+\omega(a)) &= \omega(c)+\omega(a) \\
  \eta(\omega(c)+\omega(a)) &= \omega(d)+\omega(a) \\
  \eta(\omega(d)+\omega(a)) &= \omega(b).\qedhere
 \end{align*}
\end{lem}
The next lemma is a small variation of a lemma in~\cite{Bartholdi:Growth}.
\begin{lem}
 Denote by $\partial_\omega$ the length on $G$ induced by the weight $\omega$. Then
 there are constants $C\in\NN$, $\delta<1$ such that for all $x\in S$, $g\in G$ 
 with $\partial_\omega(g)>C$ it holds 
 $\partial_\omega(p_x(g)) \leq \delta \partial_\omega(g)$.
\end{lem}
\begin{cor}
The sequences of sets
 \[\textup{Suc}_1^g=\{g\},\ \textup{Suc}_n^g = \{p_x(h) \mid h\in \textup{Suc}_{n-1}^g, x\in S\} \]
 stabilizes at a finite step for all $g\in G$.
\end{cor}
\begin{proof}[Proof of Lemma (see~{\cite[Proposition~5]{Bartholdi:Growth}})] 
 Each element $g\in G$ can be written in a word of minimal length of the form $g=a^\varepsilon x_1 a x_2 a\ldots x_n a^\zeta$ where
 $x_i\in \{b,c,d\}$ and $\varepsilon,\zeta\in \{0,1\}$. Denote by $n_b,n_c,n_d$ the number of occurrences of $b,c,d$ accordingly. 
 Then
 \begin{align*}
  \partial_\omega(g) &= (n-1+\varepsilon+\zeta)\omega(a)+n_b\omega(b)+n_c\omega(c)+n_d\omega(d)\\
  \partial_\omega(p_x(g)) &\leq (n_b+n_c)\omega(a)+n_b\omega(c)+n_c\omega(d)+n_d\omega(b) + 2\partial_\omega(x)\\
  &= \eta\left( (n_b+n_c+n_d)\omega(a)+n_b\omega(b)+n_c\omega(c)+n_d\omega(d) \right) + 2\partial_\omega(x)\\
  &= \eta(\partial_\omega(g) +(1-\varepsilon-\zeta)\omega(a)) + 2\partial_\omega(x) \\
  &\leq \eta(\partial_\omega(g)+\omega(a)) + 2(\omega(a)+\omega(b))\\
  &= \eta(\partial_\omega(g)+\omega(a)) + 2.
 \end{align*}
 Thus the length of $p_x(g)$ growths with a linear factor smaller than $1$ in terms of the length of $g$. Therefore the claim holds.
 For instance one could take $\delta =0.86$ and $C=50$ or $\delta=0.96$ and $C=16$.
\end{proof}
This completes the proof of the following proposition:
\begin{pro}\label{pro:solvableConstraintedEquations}
 If $n\geq3$ and $(g,\gamma)$ is a good pair with active constraint $\gamma$ with $\supp(\gamma)\subset\{X_1,\dotsc,X_{2n}\}$
 then the constrained equation $(R_n(X_1,\dotsc,X_{2n})g,\gamma)$ is solvable.\qed
\end{pro}
\begin{cor}
 The Grigorchuk group $G$ has commutator width at most $3$.
\end{cor}
\begin{proof}
 This is a direct consequence of the proposition and Lemma~\ref{lem:existsGoodGamma}.
\end{proof}

\subsection{Product of 2 commutators}
The case of products of two commutators can be reduced to the case of three 
commutators by using the same method as before.

We can compute the orbits of $Q^4/U_2$ and take a representative system 
denoted by $\Red^4$.
It turns out that there are $86$ orbits and we can check that there are again enough active constraints:
\begin{lem} \label{lem:existsGoodGammaForRed4}
 For each $q\in\RNK{G'}{K'}$ there is $\gamma\in \Red^4_{\act}$ such that $(q,\gamma)$ is a 
 good pair.
\end{lem}
\begin{proof}
 This can be checked in GAP with the function\newline \gapinline{verifyLemmaExistGoodGammasForRed4}.
\end{proof}

  To formulate an analog of Proposition~\ref{pro:existsNextPair} we literally transfer the definition 
  of the function $\Gamma^{q}$ to the case $n=2$. Denote the new function $\Gamma^{q,2}$.
  For a constraint $\gamma\colon F_{\mathcal{X}} \to Q$
  with nontrivial activity it produces
  a finite set $\Gamma^{q,2}(\gamma)$ of constraints $\gamma'\colon F_{\mathcal{Y}} \to Q$ 
  for an equation $(R_3g',\gamma')$.
  The role of the specialized variable $Y_{6,1}$ is now played by
  $Y_{4,1}$. As above, we denote ${\mathcal V} = \{Y_{1,1},\ldots Y_{3,2}\}$ the
  set of variables occurring in $R_3$.

\begin{pro}\label{pro:existsNextPair4}
  For each good pair $(q,\gamma)$ with $q\in\RNK{G'}{K'}$ and
  $\gamma\in\Red^4_{\act}$ the set $\Gamma^{q,2}(\gamma)$ contains
  some active constraint $\gamma'$ such that for all $g$ with $g K'=q K'$
  the pair
  $\left((g\at{2})^{\rep(\gamma'(Y_{4,1}))}\cdot g\at{1},
  \red(\gamma'|_{\mathcal{V}})\right)$ is a good pair.
\end{pro}
\begin{proof}
  The proof is the same as for
  Proposition~\ref{pro:existsNextPair}. Recalling that for
  $g\in\RNK{G'}{K'}$ we denote by $\overline g$ its image in
  $\RNK{G'}{(K\times K)}$, the corresponding formula which needs to be
  checked is
  \[\forall q\in\RNK{G'}{K'}\;
    \forall \gamma\in\Red^4_{\act}\;
    \exists \gamma'\in \Gamma^{q,2}(\gamma)\;
    \forall r\in\RNK{G'}{K'}\text{ with }\overline r=\overline p_{\rep(\gamma'(Y_{4,1}))}(q))\colon
    \mathcal{P}(q,\gamma) \Rightarrow \mathcal{P}(r,\gamma').\]
 This can be checked in GAP with the function \gapinline{verifyPropExistsSuccessor}. 
\end{proof}
The resulting succeeding pairs are now equations of genus $3$ with an active constraint. 
Those are already shown to be solvable 
by Proposition~\ref{pro:solvableConstraintedEquations}. Hence we have the following
corollary which improves Proposition~\ref{pro:solvableConstraintedEquations}:
\begin{cor}\label{cor:solvableConstraintedEquations}
If $n\geq2$ and $(g,\gamma)$ is a good pair with active constraint $\gamma$ with 
 $\supp(\gamma)\subset\{X_1,\dotsc,X_{2n}\}$
 then the constrained equation $(R_n(X_1,\dotsc,X_{2n})g,\gamma)$ is solvable. 
\end{cor}
Together with Lemma~\ref{lem:existsGoodGammaForRed4} this proves the first part of Theorem~\ref{thm:CWGrigorchukGroup}.
\begin{cor}\label{cor:KhasCW2}
 $K$ has commutator width at most $2$. 
\end{cor}
\begin{proof}
 To show that $K$ has commutator width at most $2$ it is sufficient to show that 
 the constrained equations $(R_2g,\id)$ have solutions for all $g\in K'$. 
 Since $\id$ has trivial activity one cannot directly apply 
 Proposition~\ref{pro:solvableConstraintedEquations}.  However one can check that all 
 pairs $(h,\gamma_1),(f,\gamma_2)$ such that $g=\pair{h,f}$ and
 $\gamma_1=(\id,\id,\pi(bad),\id)$, $\gamma_2=(\id,\id,\id,\pi(ca))$ are good pairs with active
 constraints and hence admit solutions $s_1,s_2\colon F_4\to G$.
 
 We can then define the map $s\colon F_4\to G, X_i\mapsto \pair{s_1(X_i),s_2(X_i)}$; it is a solution
 for $R_2g$ and $s(X_i)\in K$ for all $i=1,\dotsc,4$. Therefore the commutator width of $K$ is at most $2$.
 
 This can be checked in GAP with the function \gapinline{verifyCorollaryFiniteCWK}. 
\end{proof}

\subsection{Not every element is a commutator}
The procedure used to prove that every element is a product of two commutators 
can not be used to prove that every element is a
commutator since for equations of genus $1$ the 
genus does not increase by passing to a succeeding pair. 

In fact not every element $g\in G'$ is a commutator. This can be seen by
considering finite quotients. A commutator in the group would be also a
commutator in the quotient group. 

We will define an epimorphism to a finite group with commutator width $2$.

Analogously to the construction of $\Psi\colon \Aut(T_n )\to \Aut(T_n) \wr S_n$
we can define a homomorphism $\Psi_n\colon G \to G\wr_{2^n}(\RNK{G}{\Stab_G(n)})$ by
mapping an element $g$ to its actions on the subtrees with root in level $n$
and the activity on th $n$-th level of the tree.

Consider the following epimorphism:
\begin{align*}
  \germ\colon\left\{\begin{array}{r@{\;}l}
               G&\to\left<b,c,d\right>\simeq C_2\times C_2, \\
               a&\mapsto\id,\\
               b,c,d&\mapsto b,c,d.
             \end{array}\right.
\end{align*}
It extends to an epimorphism
$\germ_n\colon G\wr_{2^n} \RNK{G}{\Stab_G(n)}\to \germ(G)\wr_{2^n}
\RNK{G}{\Stab_G(n)}$.  We will call the image $\germ(G)=:G_0$ the
$0$-th \emph{germgroup} and furthermore
$G_n := \germ_n \circ \Psi_n(G)$ the $n$-th \emph{germgroup}.

The $4$-th germgroup of the Grigorchuk group has order $2^{26}$ and has commutator width $2$.
If the FR package is present this group can be constructed in GAP with 
the following command. 
\begin{lstlisting}
   gap> Range(EpimorphismGermGroup(GrigorchukGroup,4))
\end{lstlisting}
There is an element in the commutator subgroup of this germgroup which
is not a commutator. This element is part of the precomputed data and
can be accessed in GAP as \gapinline{PCD.nonCommutatorGermGroup4}. For the
computation of this element we used the character table of $G_4$. For more
details see Section \ref{sec:precomputation}.

 A corresponding preimage in $G$ with a minimal number of states is the
 automaton shown in Figure~\ref{fig:noncomm}. The construction of the element
 can be found in the file \filename{gap/precomputeNonCommutator.g}.
 With the representation in standard
 generators it is easy to show using the homomorphism $\pi$ on the generators
 that this element is even a member of $K$.
\begin{figure}
\includegraphics[width=\textwidth]{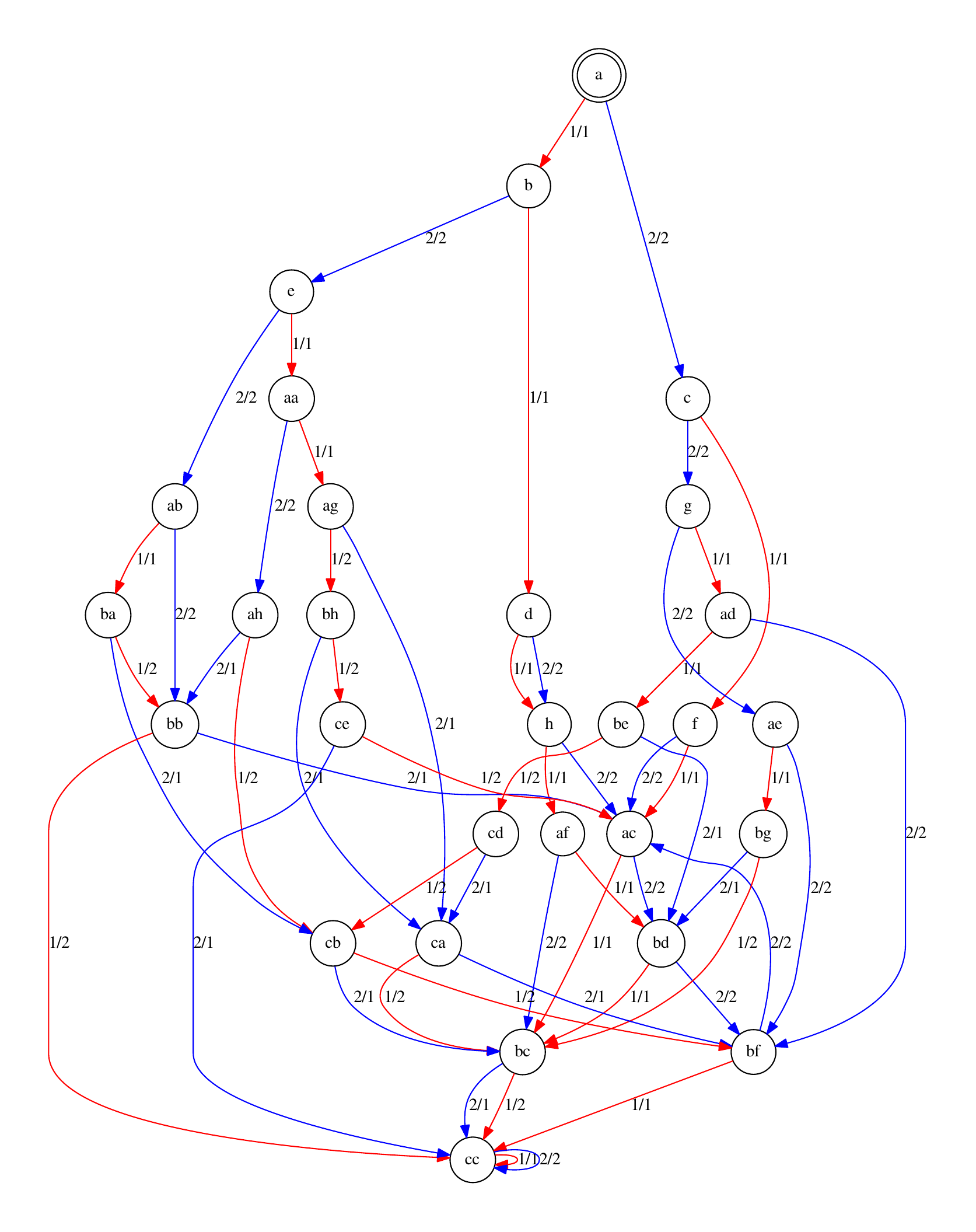}
 \caption[Noncommutator]{Element of the derived subgroup of the Grigorchuk group which is not a commutator. In standard generators:

 $(acabacad)^3\allowbreak acab(ac)^2\allowbreak (acabacad)^2\allowbreak 
 (acab)^3\allowbreak acadacab(ac)^2\allowbreak (acabacad)^2\allowbreak (acabacadacab(ac)^3\allowbreak
 abacad(acab)^2)^5\allowbreak acabacadacab(ac)^2\allowbreak (acabacad)^2\allowbreak (acabacadac)^2\allowbreak
 (abac)^3\allowbreak adacab(ac)^2\allowbreak (acabacad)^3\allowbreak acab(ac)^2\allowbreak 
 (acab(ac)^3\allowbreak abacad)^2\allowbreak acabacad((acabacadacab(ac)^2)^2\allowbreak 
 acabacad(acab)^3\allowbreak acadacab(ac)^2)^2\allowbreak ((acabacad)^3acab)^2\allowbreak 
 acab(acabacad)^2acab(ac)^2(acabacad)^3acab(ac)^3aba$}\label{fig:noncomm}
\end{figure}
This finishes the proof of Theorem~\ref{thm:CWGrigorchukGroup}.
\subsection{Bounded conjugacy width}
In~\cite{Fink:Conjugacy_growth} it is proven that $G$ has finite bounded conjugacy width. Here we give an explicit bound 
on this width.
\begin{pro}\label{pro:productOf6Conjugates}
 Let $g$ be in $G'$. Then the equation 
 \begin{align*}
  a^{X_1}a^{X_2}a^{X_3}a^{X_4}a^{X_5}ag=\id
 \end{align*}
is solvable in $G$. 
\end{pro}
\begin{proof}
We need to solve the constrained equation $(\eq{E}=a^{X_1}a^{X_2}a^{X_3}a^{X_4}a^{X_5}ag,\gamma)$ for
some constraint $\gamma$. Independently of the chosen constraint, replacement of
the variable $X_i$ by $\pair{Y_i,Z_i}\act(X_i)$ leads after normalization to an equivalent equation 
$R_2(g\at{2})(g\at{1})$. Similarly to the construction of $\Gamma^q$ in the previous section, one can find for each $q\in\RNK{G'}{K'}$ a constraint $\gamma$ such that
$\gamma(\eq{E}^{\id*\pi})=\id$ and $\gamma'\in\Gamma_1(\gamma)$ such that for all 
$g\in\pi^{-1}(q)$ the pairs $(g\at{2}g\at{1},\gamma')$ are good pairs and $\gamma'$ is 
an active constraint.
Therefore the constrained equation $(R_2(g\at{2})(g\at{1}),\gamma')$ is solvable
by Corollary~\ref{cor:solvableConstraintedEquations} 
for each $g\in G'$ and hence the equation 
$a^{X_1}a^{X_2}a^{X_3}a^{X_4}a^{X_5}ag$.
This can be checked in GAP with the function \gapinline{verifyExistGoodConjugacyConstraints}.
\end{proof}	
\begin{lem}
 There exits an element $g\in G'$ such that the equation 
 \begin{align*}
  a^{X_1}a^{X_2}a^{X_3}ag=\id
 \end{align*}
 is not solvable.
\end{lem}
\begin{proof}
As before independently of the activities of a possible constraint $\gamma$ and of the element $g\in G'$
the normalform of $\tilde\Phi_\gamma(a^{X_1}a^{X_2}a^{X_3}ag)$ turns out to
be $R_1(g\at{2}) g\at{1}$. So all there is to prove is that there is an element $h\in K$
where the products of states $h\at{2}\cdot h\at{1}$ is not a commutator.

The element $g$ displayed in Figure~\ref{fig:noncomm} provides such an element. It can easily be verified 
that $\pair{\pi(cag),\pi(ac)}\in Q_1$ and $\omega(\pair{\pi(cag),\pi(ac)})=\id$. Thus by
the properties of the branch structure we have $\pair{\pi(cag),\pi(ac)}\in K<G'$. 
\end{proof}
\begin{defi}[Conjugacy width~\cite{Fink:Conjugacy_growth}]
 The \emph{conjugacy width} of a group $G$ with respect to a generating set $S$ 
 is the smallest number $N\in\NN$ such that every element $g\in G$ is a product of
 at most $N$ conjugates of generators $s\in S$.
\end{defi}
\begin{cor}
 The Grigorchuk group $G$ with generating set $\{a,b,c,d\}$ has conjugacy width at most $8$.
\end{cor}
\begin{proof}
 The following set $T$ is a transversal of $\RNK{G}{G'}$:
 \[T=\{\id,a,d^aa,d^a,b,ab^a,ca^d,bd^a\}.\]
 Therefore, every element $g\in G$ can be written as $g=th$ with $t\in T$ and $h\in G'$. As every element of $G'$ is a product of at most $6$ conjugates of $a$ this proves the claim.

\end{proof}


This finishes the proof of Corollary~\ref{cor:productOf6Conjugates}.

\section{Proof of Theorem~\ref{thm:subgroups}}
We will prove the statement first for finite-index subgroups. 
\begin{pro}\label{pro:f.i.subgroupsFiniteCW}
  All finite-index subgroups $H \leq G$ have finite commutator width.
\end{pro}
\begin{proof}
  Note that from Corollary~\ref{cor:KhasCW2} it follows that
  $K\times K$ and furthermore $K^{\times n}$ have commutator width
  $2$.
 
  Let $H$ be a subgroup of finite index. Since $G$ has the congruence
  subgroup property (\cite{Bartholdi-Grigorchuk:parabolicSubgroups})
  we can find a nontrivial normal subgroup $N=\Stab_G(m)<H$ for some
  $m\in\NN$. Since $K$ is inactive we have $K < \Stab_G(1)$ and hence
  $K^{\times 2^n} < \Stab_G(n)$ for any $n$.
Then for every 
 subgroup $H$ of finite index there is an $n$ such that $K^{\times 2^n}\leq H$.
  
 Since $K'$ has finite index in $K$ by Lemma~\ref{lem:subgroupsOfG}, the index in $[H,H]$ of $[K^{\times2^n},K^{\times2^n}]$ is finite.
  Taking a transversal $T$ of $\RNK{[H,H]}{[K^{\times2^n},K^{\times2^n}]}$ we can find 
  $m\in \NN$ such that every element in $T$ is a product of at most $m$ commutators in $H$.
  We can thus write each element $h\in [H,H]$ as product $kt$ with $k\in K^{\times 2^n}$, 
  $t\in T$ and thus as a product of at most $2+m$ commutators.  
\end{proof}

\begin{pro} \label{pro:fgsubgroupcw}
   All finitely generated subgroups $H \leq G$ are of finite commutator width.
\end{pro}
\begin{proof}
  Every infinite finitely generated subgroup of $G$ is abstractly
  commensurable to $G$,
  see~\cite[Theorem~1]{Grigorchuk-Wilson:Commensurability}.

  This, by definition, means that every infinite finitely generated
  subgroup $H\leq G$ contains a finite-index subgroup which is
  isomorphic to a finite-index subgroup of $G$. We can repeat then the argument from the proof of Proposition~\ref{pro:f.i.subgroupsFiniteCW}.
\end{proof}

To show that there cannot be a bound on the commutator width of subgroups
we need some auxiliary results. They are well-known, but since we could not find an original reference we will
sketch their proofs here.

\begin{pro}\ 
 \begin{enumerate}
  \item For all $n\in\NN$ there is a finite $2$-group of commutator width at least $n$.
  \item $K$ contains every finite $2$-group as a subgroup.
  \item Every finite $2$-group is a quotient of two finite-index subgroups of $G$.
 \end{enumerate}
\end{pro}
\begin{proof}\
 \begin{enumerate}
  \item Consider the groups $\Gamma_n=\RNK{F_n}{\langle \gamma_3(F_n),x_1^2,\dotsc,x_n^2\rangle}$. 
  These are extensions of $C_2^n$ by $C_2^{\binom{n}{2}}$ and are class $2$-nilpotent $2$-groups.
  The derived subgroup is hence of order $2^{\binom{n}{2}}$.
  Let $T$ be a transversal of $\RNK{\Gamma_n}{\Gamma_n'}$. Thus $T$ is of order $2^n$ and
  for $x,y\in\Gamma_n$ there are $t,s\in T$ and $x',y'\in\Gamma'$ such that
  every commutator $[x,y]=[tx',sy']=[t,s]$. Therefore there are at most $\binom{2^n}{2}$ commutators.
  
  This means there are at most $\binom{2^n}{2}^m\leq 2^{(2n-1)m}$ products of $m$ commutators 
  but the size of $\Gamma_n'$ is $2^{\binom{n}{2}} \geq2^{\frac{n^2}{4}}$ and hence 
  the commutator width of $\Gamma_{8m}$ is at least $m$.
  \item  $K$ contains for each $n$ the $n$-fold iterated wreath product
  $W_n(C_2)=C_2\wr \dots \wr C_2$. 
  This can be shown by finding finitely many vertices of the tree $T_2$ which
  define a (spaced out) copy of the finite binary rooted tree with $n$ levels $T_2^n$, and finding 
  elements $k_i\in K$ such that $\left<k_i\right>$ acts on $T_2^n$ like the full group of 
  automorphisms $\Aut(T_2^n)\simeq W_n(C_2)$.
  
  Then 
  since $W_n(C_2)$ is a Sylow $2$-subgroup of $S_{2^n}$ every finite $2$-group is a subgroup of
  $W_n(C_2)$ for some $n$, and hence a subgroup of $K$.
\item Consider again some the vertices of $T_2$ which define a copy of
  the finite tree $T_2^n$ on which a subgroup of $K$ acts like
  $W_n(C_2)$. If we take $m$ large enough such that all these vertices
  are above the $m$-th level we can find a copy of $W_n(C_2)$ inside
  $\RNK{G}{\Stab_G(m)}$.\qedhere
 \end{enumerate}
\end{proof}
In the following theorem we summarize our results for the commutator width of the Grigorchuk group.
\begin{thm}\ 
 \begin{enumerate}
  \item $G$ and its branching subgroup $K$ have commutator width $2$. \label{thm:summeryStat1}
  \item All finitely generated subgroups $H\leq G$ have finite commutator width.\label{thm:summeryStat2}
  \item \label{thm:summeryStat3} The commutator width of subgroups is unbounded even among finite-index subgroups.
\item \label{thm:summeryStat4} There is a subgroup of $G$ with infinite commutator width.
 \end{enumerate}
\end{thm}
\begin{proof}
 Statements~(\ref{thm:summeryStat1}) and (\ref{thm:summeryStat2}) are proven in Theorem~\ref{thm:CWGrigorchukGroup} and Proposition~\ref{pro:fgsubgroupcw}.
 For every $n\in\NN$ we can find two groups $H_1,H_2$ of finite index in $G$ such that $\RNK{H_1}{H_2}$ has commutator
 width at least $n$. Then $H_1$ has commutator width at least $n$ as well and thus the commutator width of finite-index subgroups can not be bounded.

 For the last claim, consider a sequence $(H_i)$ of subgroups of $K$
 such that $H_i$ has commutator width at least $i$. Let
 $\psi_0\colon K\to K\times K\le K$ be the map $k\mapsto\pair{k,\id}$
 and for $i\ge1$ let $\psi_i\colon K\to K\times K\le K$ be the map
 $k\mapsto\pair{\id,\psi_{i-1}(k)}$. Then
 $H:=\langle\psi_i(H_i):i\in\NN\rangle$ is a subgroup of $K$ and hence
 of $G$ and is isomorphic to the restricted direct product of the
 $H_i$, so it has infinite width.  
\end{proof}
\section{Implementation in GAP}
\subsection{Usage of the attached files} \label{sec:gap_verify} Typing
the command \gapinline{gap verify.g} in the main directory of the
archive will produce as output a list of functions with their return
value. All these functions should return \gapinline{true}.

This approach uses precomputed data which are also in the
archive, and is very fast.

Furthermore, these data can be recomputed if a sufficiently new
version of GAP and some packages are present. For details see Section
\ref{sec:precomputation}.

This is what the functions check:
\begin{description}
\item[\texttt{verifyLemma90orbits}] This function verifies that there
  are indeed $90$ orbits of $U_3$ on $Q^6$ as claimed in
  Lemma~\ref{lem:90Constraints}.
\item[\texttt{verifyLemma86orbits}] Analogously to the previous function 
  this one verifies that there are $86$ orbits of $U_2$ on $Q^4$. 
\item[\texttt{verifyLemmaExistGoodConstraints}] This verifies that for each 
  $q\in\RNK{G'}{K'}$ there is some $\gamma\in\Red_{\act}$ such that $(q,\gamma)$ 
  forms a good pair. This is claimed in Lemma~\ref{lem:existsGoodGamma}.
\item[\texttt{verifyLemmaExistGoodConstraints4}] This is a sharper version
  of the previous function. It checks that the above statement is already true 
  if one replaces $\Red_{\act}$ by $\Red_{\act}^4$ as claimed in 
  Lemma~\ref{lem:existsGoodGammaForRed4}.
\item[\texttt{verifyPropExistsSuccessor}] This verifies that for each
  good pair
  $(q,\gamma) \in\RNK{G'}{K'}\times(\Red_{\act}\cup \Red_{\act}^4)$
  there exists a $\gamma'\in\Gamma^q(\gamma)$ such that all preimages
  of $\overline p_{\rep(Y_{6,1})}(q)$ under the map
  $\RNK{G'}{K'}\twoheadrightarrow\RNK{G'}{(K\times K)}$ form good
  pairs with the constraint $\gamma'$. This is needed in the proof of
  Proposition~\ref{pro:existsNextPair} and
  Proposition~\ref{pro:existsNextPair4}.
\item[\texttt{verifyCorollaryFiniteCWK}] Corollary~\ref{cor:KhasCW2} needs the
  existence of succeeding good pairs of the pair $(\id,\id)\in\RNK{K'}{K'}\times\Red^4$.
  This function verifies this existence. 
\item[\texttt{verifyExistGoodConjugacyConstraints}] This verifies that for the equation
  $a^{X_1}a^{X_2}a^{X_3}a^{X_4}a^{X_5}a$ there are constraints $\gamma$ 
  that admit good succeeding pairs. This is needed in the proof of
  Proposition~\ref{pro:productOf6Conjugates}.
\item[\texttt{verifyGermGroup4hasCW}] This function verifies the existence of 
  an element in the derived subgroup of the $4$-th level germgroup that is not a commutator.
\end{description}

\subsection{Precomputed data}\label{sec:precomputation}
In the interactive gap shell started by \gapinline{gap verify.g}
the precomputed data is read from some files in \filename{gap/PCD/} and stored in 
a record \gapinline{PCD}. 

One can use the function \gapinline{RedoPrecomputation} with one argument. In each case
the result is written to one ore multiple files and will override the original precomputed data. 
The argument is a string and can be one of the following:
\begin{description}
    \item [\texttt{``orbits''}] This computes the $90$ orbits 
    of $\RNK{\Aut(F_6)}{U_3}$ and the
		$86$ orbits of $\RNK{\Aut(F_4)}{U_2}$. This computation will take
		about $12$ hours on an ordinary machine and has no progress bar.
   \item [\texttt{``goodpairs''}] First this computes for each constraint $\gamma\in\Red\cup\Red^4$ 
		      the set of all $q\in\RNK{G'}{K'}$ such that $(q,\gamma)$ is a good pair.
		      
		      Then it computes for each good pair $(q,\gamma)$ one 
		      $\gamma'\in\Gamma^q(\gamma)$
		      with decorated $X=Y_{6,1}$ or $X=Y_{4,1}\in S$ as defined in equation~
          \eqref{def:Gammaq} which
		      fulfills depending whether $\gamma\in\Red^4_{\act}$ or $\gamma\in\Red_{\act}$ 
		      either Proposition~\ref{pro:existsNextPair} or Proposition~\ref{pro:existsNextPair4}.
		      This computation takes about half an hour on ordinary machines
           and is equipped with a progress bar. 
		      
		      Afterwards the succeeding pairs of $(\id,\id)$ which are 
          needed for Corollary~\ref{cor:KhasCW2} are computed. 
   \item [\texttt{``conjugacywidth''}] Denote by $\eq{E}_g$ the equation $a^{X_1}a^{X_2}a^{X_3}a^{X_4}a^{X_5}ag$.
		      Letting $q\in G/K'$ be the image of $g$, this computes a constraint 
		      $\gamma\colon F_5 \to Q$ for the equations $\eq{E}_g$
		      and a constraint $\gamma'\colon F_4\to Q$ such that
		      $(\gamma * \pi)(\eq{E}_g) = \id$,
		      \[\eq{E}'_g := \nf(\tilde\Phi_\gamma(\eq{E}_g))=[X_1,X_2][X_3,X_4](g
                        \at{2})(g\at{1}),\] and $(\eq{E}'_g,\gamma')$
                      is a good pair for all $g$ with $g K'=q$.
		      
		      The computation takes about one hour and is equipped with a progress bar.
   \item [\texttt{``charactertable''}] This computes the character table of 
          the $4$-th level germgroup and the set of irreducible characters. 
          As the germgroup is quite large, this
          takes about $3$ hours. There is no kind of progress bar.
   \item [\texttt{``noncommutator''}] Inside the $4$-th level germgroup there is
          an element which is not a commutator but in the commutator subgroup.
          Since this group is finite we could in principle search by brute force
          for a commutator. 
          Luckily there are only $3106$ irreducible characters in this group and
          therefore we can use Burnside's formula \eqref{eq:BurnsideFormula}. 
          The search will almost immediately give a result. Most of the
          computation time is used to assert that the found element is indeed
          not a commutator.
		      
	  The element is then lifted to its preimage in $G$ with a minimal 
	  number of states.

	  Checking the assertion takes approximately $3$ hours and is equipped 
	  with a progress bar. 
  \item [\texttt{``all''}] This performs all of the above one after another.		      
   \end{description}
To recompute the orbits or the charactertable GAP should be 
started with the \gapinline{-o} flag
to provide enough memory for the computation. For example start GAP by 
\gapinline{gap -o 8G verify.g}

\subsection{Implementation details}
\label{sec:gap_details}
\subsubsection{Reduced Constraints}
The proof of Lemma~\ref{lem:finitelyManyConstraints} in~\cite{Lysenok-Miasnikov-Ushakov:QuadraticEquationsInGrig} 
provides a constructive method to reduce any constraint to one with support
only in the first five variables. 
We have implemented this in the function \gapinline{ReducedConstraint} in the file
\filename{gap/functionsFR.g}.

It uses that the quotient $Q=\RNK{G}{K}$ is a polycyclic group with 
\begin{align*}
 C_0 &= Q = \left<\pi(a),\pi(b),\pi(d)\right>, &
 C_1 &= \left<\pi(a),\pi(d)\right>, &
 C_2 &= \left<\pi(ad)\right>.
\end{align*}
We take the generators of $U_n$ as given in the proof of
 Lemma~\ref{lem:90Constraints}
plus additional ones which switch two neighboring pairs:
\begin{align*}
  s_i &\colon\begin{array}{r@{\;}l}
               X_i&\mapsto X_{i+2}\\
               X_{i+1}&\mapsto X_{i+3}\\
               X_{i+2}&\mapsto X_{i}^{[X_{i+2},X_{i+3}]}\\
               X_{i+3}&\mapsto X_{i+1}^{[X_{i+2},X_{i+3}]}
             \end{array}
                   \textup{ for }i = 1,3,\dotsc,2n-3.
\end{align*}
It can easily be checked, that these are also contained in $U_n$. 
These elements are used to reduce a given constraint in a form of 
a list with entries in $Q$ to a
list where all entries with index larger then $5$ are trivial. 
This constraint can then be further reduced by a lookup table for the orbits
of $\RNK{\Aut(F_6)}{U_3}$. 

If the file \filename{verify.g} is loaded in a GAP environment with the FR package available 
the function \gapinline{ReducedConstraint} can be used as an alias to get 
reduced constraints. For example:
 \begin{lstlisting}
    gap> f1 := Q.3;
    gap> gamma:= [f1,f1,f1,f1,f1,f1];
    gap> constr := ReducedConstraint(gamma);;
    gap> Print(constr.constraint);
\end{lstlisting} 
\begin{verbatim}
[ <id>, <id>, <id>, <id> , f1, <id>]
\end{verbatim} 

\subsubsection{Good pairs}
For $g\in G$ and a constraint $\gamma$ the question whether $(g,\gamma)$ is a good
pair depends only on the image of $g$ in $\RNK{G}{K'}$ and the representative
of $\gamma\in\Red$. (See Section~\ref{sec:good_pairs}.) So this is already a
finite problem. 

Given a given constraint $\gamma$, to obtain all $q$ which form a good
pair we can enumerate all possible commutators
$[r_1,r_2][r_3,r_4][r_5,r_6]$ with $r_i K=\gamma(X_i)$.
Since $\lvert\RNK{K}{K'}\rvert=64$, it would take too much time to
consider all combinations at once; thus the possible values for
$[r_1,r_2]$ are computed and in a second step triple products of those
elements are enumerated.  This is implemented in the function
\gapinline{goodPairs} in the file \filename{gap/functions.g}.

\subsubsection{Successors}
The key ingredient for the proof of
Theorem~\ref{thm:CWGrigorchukGroup} is
Proposition~\ref{pro:existsNextPair}. The main computational effort
there is to compute the sets $\Gamma_q(\gamma)$ and find good pairs
inside them.

This is implemented exactly as explained in the construction of the map
$\Gamma_q$ in the function \gapinline{GetSuccessor} in the file 
\filename{gap/precomputeGoodPairs.g}. Given an element $q\in\RNK{G'}{K'}$ 
and an active constraint $\gamma$ this function returns a tuple $(\gamma',X)$ 
with $\gamma\in\Red$ and $X$ the decorated element $Y_{6,1}$ or $>Y_{4,1}$
depending if $\gamma\in\Red^4$ or $\gamma\in\Red$.

Given an inactive constraint $\gamma$ it returns a pair of constraints
$\gamma_1,\gamma_2$ such that both have nontrivial activity and with
$\omega$ the map from the branch structure it holds:
$\omega(\pair{\gamma_1(X_i),\gamma_2(X_i)})=\gamma(X_i)$. 

If the FR package is available
the function \gapinline{GetSuccessorLookup} can be used to explore the
successors of elements. It returns the succeeding pair. For example
 \begin{lstlisting}
    gap> f4 := Q.1;
    gap> gamma:= [f4,f4,f4,f4,f4,f4];;
    gap> g := (a*b)^8;;
    gap> IsGoodPair(g,gamma);
\end{lstlisting}
\begin{verbatim}
true
\end{verbatim} 
\begin{lstlisting}
    gap> suc := GetSuccessorLookup(g,gamma);;    
    gap> suc[1];
\end{lstlisting}
\begin{verbatim}
<Trivial Mealy element on alphabet [ 1 .. 2 ]>
\end{verbatim} 
\begin{lstlisting}
    gap> suc[2].constraint;
\end{lstlisting}
\begin{verbatim}
[ <id>, <id>, <id>, <id>, f1*f3, <id> ]
\end{verbatim} 

\section*{Acknowledgments}
The authors are deeply grateful to Rachel Skipper for her remarks that
helped improve the presentation of this material.


\bibliography{latex/bio}

\end{document}